\documentclass[preprint,12pt]{elsarticle}

\usepackage{amsmath,amsfonts,amsthm,amssymb,tikz,enumerate,enumitem, hyperref, graphicx, marvosym, listings, breqn, adjustbox, tabularx, footnote}
\usepackage{geometry} 
\usepackage{setspace}
\usepackage{pgf, tikz, subcaption, caption}
\usepackage{threeparttable} 
\pgfdeclarelayer{background}
\pgfsetlayers{background,main}

\newtheorem{conjecture}{Conjecture}

\newtheorem{definition}{Definition}

\newtheorem{lemma}{Lemma}

\newtheorem{theorem}{Theorem}
\numberwithin{lemma}{section}
\numberwithin{equation}{section}
\numberwithin{theorem}{section}
\numberwithin{definition}{section}
\numberwithin{conjecture}{section}
\numberwithin{corollary}{section}
 \newcommand{\Hsquare}{%
  \text{\fboxsep=-.2pt\fbox{\rule{0pt}{1ex}\rule{1ex}{0pt}}}%
}

\theoremstyle{definition}

\usetikzlibrary{arrows, automata}
\tikzstyle{vertex}=[circle,draw,minimum size=12pt,inner sep=0pt]
\tikzstyle{selected vertex} = [vertex, fill=red!24]
\tikzstyle{edge} = [draw,thick,-]
\tikzstyle{weight} = [font=\small]
\tikzstyle{cycle edge} = [draw,line width=5pt,-,red!50]
\tikzstyle{tail edge} = [draw,line width=5pt,-,blue!50]
\tikzstyle{tree edge} = [draw,line width=3pt,-,black!20]
\usetikzlibrary{positioning}

\journal{Discrete Applied Mathematics}

\begin{document}

\begin{frontmatter}



\title{Automating Weight Function Generation in Graph Pebbling}


\author[umd]{Dominic Flocco}
\author[dc]{Jonad Pulaj}
\author[dc]{Carl Yerger}
\affiliation[umd]{organization={Department of Mathematics, University of Maryland},
            city={College Park},
            postcode={20742}, 
            state={MD},
            country={USA}}

\affiliation[dc]{organization={Department of Mathematics and Computer Science, Davidson College},
            city={Davidson},
            postcode={28035}, 
            state={NC},
            country={USA}}

\begin{abstract}
\indent Graph pebbling is a combinatorial game played on an undirected graph with an initial configuration of pebbles. A pebbling move consists of removing two pebbles from one vertex and placing one pebble on an adjacent vertex. The pebbling number of a graph is the smallest number of pebbles necessary such that, given any initial configuration of pebbles, at least one pebble can be moved to a specified root vertex. Recent lines of inquiry apply computational techniques to pebbling bound generation and improvement.  \\ \indent Along these lines, we present a computational framework that produces a set of tree strategy weight functions that are capable of proving pebbling number upper bounds on a connected graph. Our mixed-integer linear programming approach automates the generation of large sets of such functions and provides verifiable certificates of pebbling number upper bounds. The framework is capable of producing verifiable pebbling bounds on any connected graph, regardless of its structure or pebbling properties. We apply the model to the 4\textsuperscript{th} weak Bruhat to prove $\pi(B_4) \leq 66$ and to the Lemke square graph to produce a set of certificates that verify $\pi(L\Hsquare L) \leq 96$.
\end{abstract}



\begin{keyword}


Pebbling\sep graph theory \sep integer programming \sep combinatorial optimization
\end{keyword}

\end{frontmatter}


\section{Introduction}
\label{sec:intro}
Graph pebbling is a combinatorial game played on an undirected graph with an initial configuration of pebbles. The game is composed of a sequence of pebbling moves that consist of removing two pebbles from one vertex, and placing one pebble on an adjacent vertex. The graph pebbling model was first introduced by Chung \cite{Chung1989} and can be used to model transportation networks where resources are consumed in transit, such as the dissipation of energy, the flow of information through a particular medium or the consumption of fuel by oil tankers. Thus, graph pebbling may be viewed as a network optimization model that has applications to operations research, logistics and applied science. The central question of the model asks how many pebbles are necessary to ensure that, given any initial configuration of this many pebbles, one pebble can be moved to any particular vertex. \\ 
\indent Conventional graph pebbling techniques employ topics in number theory \cite{Chung1989, Herscovici2008, Pachter1995, Wang2000}, combinatorics \cite{Czygrinow2002, Moews1992} and graph theory \cite{Blasiak2008, Bukh2006, Postle2013}. Since the conceptualization of graph pebbling in 1989 \cite{Chung1989}, there have been many interesting and useful results presented through these theoretical mediums, particularly on small graphs and families of graphs. In fact, the graph pebbling model was brought forth by Lagarias and Saks in response to a combinatorial number theory conjecture on zero-sum problems by Erdos and Lemke. These methods are successful in defining properties that provide sufficient grounds for pebbling classifications, and producing general results with diverse applications. \\ 
\indent A recent line of inquiry in the graph pebbling community explores computational approaches to calculating pebbling bounds. Calculating pebbling numbers is difficult: Watson \cite{Watson2005} shows that determining whether a configuration of pebbles is solvable is NP-complete, and Clark and Milans \cite{Milans2006} prove that deciding whether the pebbling number is at most $k$ is $\Pi^P_2$-complete. To address this complexity, we apply methods in linear and discrete optimization to the graph pebbling model, building off the work of \cite{Cranston2015}, \cite{Hurlbert2010} and \cite{kenter2018}. Namely, we use mixed-integer linear programming (MILP) to calculate pebbling bounds, building upon theoretical results and improving bounds on graphs of interest. Through exploring popular problems in graph pebbling literature, we develop a computational approach that can be applied to general networks. This computational framework is packaged into an open source graph pebbling codebase that includes MILP implementation, verification and visualization tools, and pebbling graph classifications. The modeling suite is complete with a user-interface and documentation to allow for easy experimentation by future researchers.\footnote{\url{https://github.com/dominicflocco/Graph\_Pebbling}}
\subsection{Notation \& Terminology}
\label{subsec:notation}
Let $G= (V,E)$ be a simple, connected and undirected graph with vertex set $V= V(G)$ and edge set $E=E(G)$, with $|V|= n$ vertices and $|E|=m$ edges. Define $\text{dist}(x,y)$ to be the number of edges (length) of the shortest path from vertex $x$ to $y$. The \textit{eccentricity} $\text{ecc}(G,u)$ for a vertex $u\in V$ is given by $\text{ecc}(G,u) = \max_{v\in V}\text{dist}(v,u)$, and the \textit{diameter} $\text{diam}(G)$ of $G$ is then given by $\text{diam}(G) = \max_{u\in V}\text{ecc}(G, u)$. We will also consider the Cartesian product of two graphs, denoted $\Hsquare$, where the product of $G_1=(V_1, E_1)$ and $G_2=(V_2, E_2)$ is defined by the vertex set $V(G_1\Hsquare G_2) = \{(v_1, v_2) : v_1 \in V(G_1), v_2\in V(G_2)\}$ and edge set \begin{align} E(G_1\Hsquare G_2) = \{\{(v_1, v_2), (u_1,u_2)\} : &(v_1 = u_1 \text{ and } (v_2,u_2)\in E(G_2)) \\ &\text{ or } (v_2 = u_2 \text{ and } (v_1, u_1) \in E(G_1))\}. \nonumber \end{align} \\ 
\indent For a graph $G= (V, E)$, a \textit{configuration} of pebbles is a function $C : V \rightarrow \mathbb{Z}_{\geq 0}$. Let $C(v)$ be the number of pebbles placed on vertex $v\in V$, and define the \textit{size} of the configuration as $|C| = \sum_{v\in V}C(v)$, which is the number of pebbles placed on all vertices in $G$. A \textit{pebbling move} from $u$ to $v$ along an edge $(u,v) \in E$ consists of removing two pebbles from vertex $u\in V$ and placing one pebble at adjacent vertex $v\in V$. We say that a pebble can be \textit{moved} to the \textit{root} vertex $r$ if we can repeatedly apply pebbling steps to $G$ such that $r$ has one pebble in the resulting configuration. 
\\ \indent Consider a graph $G$ rooted at $r$ with a configuration of pebbles $C$. We define the \textit{rooted pebbling number} $\pi(G,r)$ as the smallest integer $k$ such that for any configuration of $|C|=k$ pebbles onto the vertices of $G$, one pebble can be moved to the specified root vertex $r$. The \textit{pebbling number} $\pi(G)$ is the smallest integer $k$ such that for any configuration of $k$ pebbles onto the vertices of $G$, one pebble can be moved to each $r\in V(G)$. Given a configuration of pebbles $C$, if it is possible to move a pebble to the root vertex $r$ by a sequence of pebbling moves, then we say that $C$ is $r$-\textit{solvable}. If not, $C$ is $r$-\textit{unsolvable}. If $C$ is $r$-solvable for each $r\in V$ then we say $C$ is \textit{solvable}, and \textit{unsolvable} otherwise. For other important graph theoretic concepts and definitions, the reader is referred to \cite{West2001}.\\ 
\subsection{Results}
\label{subsec:intro-results}
This paper explores the use of linear and mixed-integer linear programming to prove pebbling number bounds on a variety of graphs of interest. We offer the theoretical framework, implementation and formal definition of two MILPs that build on the foundation set by Cranston et al. \cite{Cranston2015} and Hurlbert \cite{Hurlbert2010}. Our optimization problems generate larger, more useful sets of weight functions that improve pebbling bounds previously found with manually generated strategies. Namely, we present the Tree Strategy Optimization Problem $\mathbf{TS}_{G,r}$, which uses mixed-integer linear programming to generate sets of tree strategies that minimize the pebbling number bound. Then, we leverage the symmetry of tree strategies in Cartesian product graphs to reduce the number of constraints in $\mathbf{TS}_{G,r}$ in the Symmetric Tree Strategy Optimization Problem $\mathbf{STS}_{G,r}$. Together, these optimization problems produce useful weight functions that have the potential to solve a plethora of graph pebbling models and have proven successful in improving pebbling bounds on a number of graphs of interest in the literature, such as the Lemke square $L\Hsquare L$ and 4\textsuperscript{th} weak Bruhat $B_4$. \\ 
\indent The optimization problems presented in our paper show significant promise in improving pebbling bounds on a number of important families of graphs. Most notably, we prove that $\pi(L\Hsquare L) \leq 96$, improving the upper bound of $108$ found in \cite{Hurlbert2010}. Kenter \cite{kenter2018, Kenter2019} uses an integer programming approach to numerically suggest evidence that $\pi(L\Hsquare L)\leq 91$ and then improves this bound to $85$. While these results are impressive, our mixed-integer programming framework not only provides certificates that can be checked by hand but also bypasses the chance of error imposed by floating-point arithmetic possible in Kenter's approach, a problem that Cook \cite{Cook2013} shows possible. Additionally, the tree strategy visualizations generated by our MILP approach can provide intuition into the underlying structures of the graphs and insight into their pebbling properties. We also prove that $\pi(B_4) \leq 66$ using only tree strategies, further improving the results of Hurlbert \cite{Hurlbert2010}. Cranston et al. \cite{Cranston2015} show that this bound can be improved to 64 using using non-tree weight functions.\\ 
\indent The paper is organized as follows. Section \ref{sec:background} presents the necessary graph pebbling theoretical framework upon which our mixed-integer programs rely. Then, Section \ref{sec:compute} offers the formal definition of the MILP framework and provides a theoretical discussion on the model. Lastly, Section \ref{sec:results} displays important and novel results obtained through these techniques, and   Section \ref{sec:conclusion} renders concluding remarks.

\tikzstyle{vertex}=[circle,fill=black!10,minimum size=12pt,inner sep=0pt]
\tikzstyle{selected vertex} = [vertex, fill=red!24]
\tikzstyle{edge} = [draw,thick,-]
\tikzstyle{weight} = [font=\small]
\tikzstyle{cycle edge} = [draw,line width=3pt,-,red!50]
\tikzstyle{tail edge} = [draw,line width=3pt,-,blue!50]
\tikzstyle{tree edge} = [draw,line width=3pt,-,black!50]

\section{Background} \label{sec:background}
The following discussion introduces theoretical results that lay the foundation for the novel linear optimization approach presented in Section \ref{sec:compute}. In addition, we survey the existing literature on computational graph pebbling techniques, and discuss how our work relates these recent efforts. 
\subsection{Graph Products}
In the initial formalization of the graph pebbling problem in Chung \cite{Chung1989} the reader is left with a conjecture by Ronald Graham. Since its proposition, Graham's conjecture has been of great interest to the pebbling community.
\begin{conjecture}[Graham] For all graphs $G$ and $H$, we have $\pi(G\Hsquare H)\leq \pi(G)\pi(H)$. 
\label{conj:graham}
\end{conjecture}
While many useful results in support of this conjecture have been offered, the conjecture has yet to be proven for all graphs. The conjecture is verified when $G$ and $H$ are trees \cite{Moews1992}, cycles \cite{Herscovici2008}, a clique and a graph with the 2-pebbling property \cite{Chung1989} and when $G$ has the 2-pebbling property and $H$ is a complete multi-partite graph \cite{Wang2000}. Feng \cite{Feng2001} also verified the conjecture for complete bipartite graphs and wheels or fans. Later, Hurlbert \cite{Hurlbert2000} proves Graham's conjecture for a more general product. While these results are encouraging, there is evidence to suggest that the Lemke graph $L$, shown in Figure \ref{fig:lemke-template}, may provide a counterexample to the conjecture, and computational methods could be key to finding a counterexample. Note that $\pi(L, r) = 8$ for all $r\in V(L)$, so $\pi(L)\pi(L) = 64$; thus, if one can show $\pi(L\Hsquare L) > 64$; Graham's conjecture will be disproved. As a result, our work focuses on the application of the MILP computational approaches to weight function generation to the Cartesian product of graphs.
\\ 
\begin{figure}
    \centering
    \begin{tikzpicture}[scale=0.5, auto,swap, vertex/.style={circle,draw,minimum size=6mm,inner sep=0}]
    \foreach \pos/\name/\label in {{(0,0)/v_6/}, {(0,4)/v_5/}, {(0,-4)/v_7/}, {(4,2)/v_3/}, {(8,2)/v_1/}, {(4,-2)/v_4/}, {(8,-2)/v_2/}, {(-4,0)/v_8/}}
        \node[vertex, label={\small $\label$}] (\name) at \pos {$\name$};
    \foreach \source/ \dest in {{v_8/v_5}, {v_8/v_7}, {v_8/v_6}, {v_7/v_3}, {v_7/v_4}, {v_6/v_3}, {v_6/v_4}, {v_5/v_3}, {v_5/v_4}, {v_4/v_2}, {v_3/v_1}, {v_2/v_1}, {v_8/v_4}}
        \path[edge] (\source) -- (\dest);
\end{tikzpicture}
    \caption{The first minimal non-isomorphic Lemke graph $L$ on $n=8$ vertices. }
    \label{fig:lemke-template}
\end{figure}
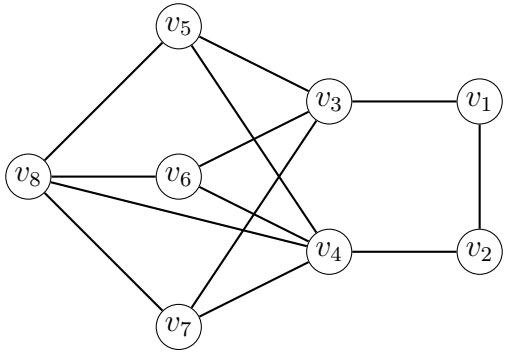

\subsection{Computational Approaches} 
\label{subsec:back-comp-approach}
Due to the difficulty in calculating tight pebbling bounds, the pebbling literature has seen a recent rise in computational pursuits. The algorithmic approach has been successful in computing pebbling numbers and bounds on graphs with useful characteristics. Bekmetjev and Cusack \cite{Bekmetjev2009} explore algorithms to determine the solvabilility of pebbling configurations on graphs with diameter two, proving the existence of an algorithm whose running time depends on the vertex connectivity and the size of the graph. Sieben \cite{Sieben2010} extends this approach and presents an algorithm that computes the pebbling number on all connected graphs with up to 9 vertices with diameter 2. The implementation is sufficiently fast on a large number of small graphs, yet does not scale well as the number of vertices increase. While the algorithmic approach has proven successful for small graphs with certain properties, other promising computational approaches have recently appeared in the literature. \\ 
\indent Recently, Kenter \cite{kenter2018, Kenter2019} presents an integer-programming (IP) approach to bounding $\pi(L\Hsquare L)$ and Cartesian products more generally. This approach is specific to $L\Hsquare L$, yet can be expanded to provide bounds on the pebbling numbers of different types of graph products. The model leverages the 2-pebbling property of $L$ to refine a more general IP approach that offers numerical evidence to suggest $\pi(L\Hsquare L)\leq 85$. In contrast, our MILP approach chooses an ideal set of the exponentially many subtrees in a graph to generate weight functions that offer verifiable certificates of pebbling bounds. \\ 
\indent In other recent efforts, Cusack et al. \cite{Cusack2018} explore the two-pebbling property more deeply and introduce a novel method to create infinite families of Lemke graphs. Additionally, their work describes configurations that violate the two-pebbling property on Lemke graphs, and the reader is left with a conjecture on a more general infinite family of Lemke graphs. The authors extend this work in \cite{Cusack2019} to provide a discussion on computational approaches to solvability and reachibility, and present an algorithm that leverages the two-pebbling property. The technique enumerates all minimally violating configurations and is employed to find all Lemke graphs on 9 or fewer vertices. In our work, we add to this growing literature of computational graph pebbling approaches by presenting a general technique to pebbling bound generation.

\subsection{Weight Functions} 
\label{subsubsec:back-weight-funcs}
\indent Linear optimization techniques have proven successful in improving pebbling number bounds on large graphs and represent a state-of-the-art computational approach to pebbling. Hurlbert \cite{Hurlbert2010} first introduces a linear optimization problem to compute upper bounds on pebbling numbers and lays the foundation for the following investigation. The foundational tool used in the linear optimization problem is Weight Function Lemma \ref{lem:weight-func}.\\
\indent Let $G$ be a graph and $\mathcal{T}$ be a subtree of $G$ rooted at $r$. For each vertex $i\in V(\mathcal{T}) \backslash \{r\}$, let $i^+$ denote the \textit{parent} of $i$, where $\text{dist}(i^+, r) = \text{dist}(i, r) - 1$. 
\begin{definition}[Tree Strategy] 
Let $G = (V,E)$ be an undirected graph and let $r\in V(G)$ be an arbitrary root vertex. We call $\mathcal{T} = (V, E, w)$, with vertex set $V(\mathcal{T})$, edge set $E(\mathcal{T})$ and nonnegative weight function $w: V(\mathcal{T}) \rightarrow \mathbb{R}_{\geq 0}$, a valid tree strategy when $w$ satisfies the following three properties: (1) $w(r) = 0$, (2) $w(i^+) \geq 2w(i)$ for all $i\in V(\mathcal{T}) \backslash\{r\}$ that are not adjacent to $r$, and (3) $w(i)=0$ if $i\notin V(\mathcal{T})$. 
\label{def:tree-strategy}
\end{definition} 
Note that the weight function $w$ may also be viewed as a vector $w\in \mathbb{R}^{n}_{\geq 0}$ representing the $n= |V(G)|$ weights of the tree strategy, where $w_i$ corresponds to the weight on $v_i$. We can also define a configuration of pebbles $C\in \mathbb{Z}^{n}_{\geq 0}$ in a similar manner and let $\mathbf{1}_G$ be the vector on $V(G)$ in which every entry is 1. Weight Function Lemma \ref{lem:weight-func} permits the consideration of all subtrees rooted at $r$ simultaneously by utilizing the conic combination of strategies. We will see that the ability to simultaneously consider all subtrees provides a more powerful and versatile approach to pebbling upper bound calculation.
\begin{lemma}[Weight Function Lemma] \cite{Hurlbert2010}
Let $\mathcal{T}$ be a tree strategy of $G$ rooted at $r$, with associated weight function $w$. If $C$ is an $r$-unsolvable configuration of pebbles on $V(G)$, then $w\cdot C \leq w \cdot \mathbf{1}_G$.
\label{lem:weight-func}
\end{lemma}
The proof idea is rather intuitive. If we suppose that $C$ is a configuration such that $w\cdot C > w \cdot \mathbf{1}_G$, there must be some vertex $v\in V(\mathcal{T})$ that has at least two pebbles. Note that if $v = r$ or $v$ is adjacent to $r$, then the proof is trivial. Hence, we assume $\text{dist}(v,r)\geq 2$. Now, make a pebbling move from $v$ towards the root to $v^+$ and denote the resulting configuration $C'$. By the definition of a weight function $w$, we have $w(v^+) \geq 2w(v)$, so $w\cdot C' \geq w\cdot C >  w\cdot\mathbf{1}_G $. Repeating this process we can eventually move a pebble to root $r$. It follows that for every graph $G$ and root vertex $r$, the pebbling number is bounded above by one plus the number of pebbles in the largest configuration $C$ such that $w\cdot C \leq w \cdot \mathbf{1}_G$. \\ 
\subsection{Linear Programming Approach} 
\label{subsec:lp-approach}
\indent The Weight Function Lemma \ref{lem:weight-func} allows the definition of a linear optimization problem capable of computing bounds on large graphs and opens the door to a new computational approach with the potential of generating impressive results on many graphs of interest. Suppose we have $T$ tree strategies indexed $t=1,2, \ldots, T$ in graph $G$ rooted at $r$. We now arrive at the following integer linear program (LP) first offered by Hurlbert \cite{Hurlbert2010}: 
\begin{subequations} \label{eq:lin-opt}
\begin{align}
\label{eq:lin-optim1}
\text{Maximize:}& \sum_{v\in V \backslash \{r\}} C(v)\\ \text{Subject to: }&  w_t \cdot C \leq w_t \cdot \mathbf{1}_G \qquad t = 1, \ldots, T.
\label{eq:lin-optim2}
\end{align}
\end{subequations}
Objective function (\ref{eq:lin-optim1}) maximizes the size of the configuration, while constraint (\ref{eq:lin-optim2}) ensures that the configuration is not trivially solvable by Lemma \ref{lem:weight-func}. Let $z_{G,r}$ be the optimal value of (\ref{eq:lin-opt}) and let $\hat{z}_{G,r}$ be the optimum of the linear relaxation. Then, $z_{G,r}\leq \lfloor \hat{z}_{G,r}\rfloor$ and it follows from Lemma \ref{lem:weight-func} that $\pi(G,r) \leq z_{G,r} + 1 \leq \lfloor\hat{z}_{G,r}\rfloor + 1$. Therefore, LP (\ref{eq:lin-opt}) can be used to prove the upper bound on pebbling numbers with intelligently designed weight functions. If we let $w'$ be a conic combination of the associated weight functions for each of the $T$ tree strategies $w_1, \ldots, w_T$, then this idea generalizes to the following Lemma. 
\begin{lemma}[Covering Lemma]\cite{Cranston2015} Given a graph $G$ rooted at $r\in V(G)$, let $w'$ be a conic combination of tree strategies for $r$, and let $K$ and $M$ be positive constants. If $w'(v)\geq K$ for all $v\in V(G) \backslash \{r\}$ and $\sum_{v\in V(G)\backslash \{r\}} w'(v) < M$, then $\pi(G,r) \leq \left\lfloor\frac{M}{K}\right\rfloor + 1$. 
\label{lem:covering}
\end{lemma}
For a conic combination of tree strategies $w'$, the constant $K$ which yields the tightest upper bound on $\pi(G,r)$ is 
\begin{equation}
    K = \min_{v\in V(G)\backslash\{r\}} w'(v) 
    \label{eq:K}.
\end{equation}
Hence, $K$ is the minimum vertex weight over all tree strategies; and, $M$ is the sum of vertex weights over all tree strategies plus some $\epsilon >0$. Namely, 
\begin{equation}
    M = \left[\sum_{v\in V(G)\backslash \{r\}} w'(v)\right] + \epsilon,
\end{equation}
for $\epsilon >0$. In practice, $M$ is quite large, so $\epsilon = 1$ is the conventional choice for pebbling bound calculation with Lemma \ref{lem:covering}.
Defined in this way, the value $\left\lfloor \frac{M}{K}\right\rfloor$ can be interpreted as the average tree weight across all strategies. The difficulty in using Lemma \ref{lem:covering} to produce a low upper bound on $\pi(G,r)$ is in choosing a small subset of the potentially exponential subtrees of a graph to generate the best possible pebbling bound. The approach to this problem is unclear and will be the focus of our work. 
 \\ 
\indent Hurlbert \cite{Hurlbert2010} uses this approach to obtain results on many families of graphs, namely trees, cycles, cubes, Petersen, Lemke and 4\textsuperscript{th} weak Bruhat. Cranston et al. \cite{Cranston2015} build upon this technique and generalize the set of weight functions beyond trees. This generalization proves successful in improving the bounds on the Lemke graph, 3-cube and 4\textsuperscript{th} weak Bruhat graph from \cite{Hurlbert2010}. The Weight Function Lemma \ref{lem:weight-func}, Covering Lemma \ref{lem:covering}, and the accompanying linear optimization problem (\ref{eq:lin-opt}), introduced and explored in \cite{Hurlbert2010} and \cite{Cranston2015}, prove computationally tractable and successful in computing pebbling bounds on large graphs. In our research, we explore and extend this optimization approach and further exhibit its versatility and strengths by automating the generation of weight functions.

\section{Mixed-Integer Linear Programming Approach}
\label{sec:compute}
\indent We introduce a tree strategy optimization problem to automate the task of tree generation on a given graph $G$ rooted at $r$. As there are an exponential number of subtrees in $G$, our task is to choose an \textit{ideal} set of strategies that produce low pebbling bounds using Covering Lemma \ref{lem:covering}. To do so, we fix the number of tree strategies $T$ as a parameter to the linear optimization problem and index tree strategies from the ordered set $\{1,2, \ldots , T\} = [T]$. Further, we define parameter $\ell$ to be an upper bound on the tree diameter, formally defined as $\ell =\max_{v\in \mathcal{T}}\text{dist}(r,v)$, or the maximum eccentricity of any vertex in the tree. We wish to minimize the average tree weight to obtain a small pebbling number bound. Thus, we seek a set of tree strategies that minimize the average tree weight $M$ while requiring $K \geq T$. Hence, we assume that the vertex with minimum weight has average weight of one. \\ 

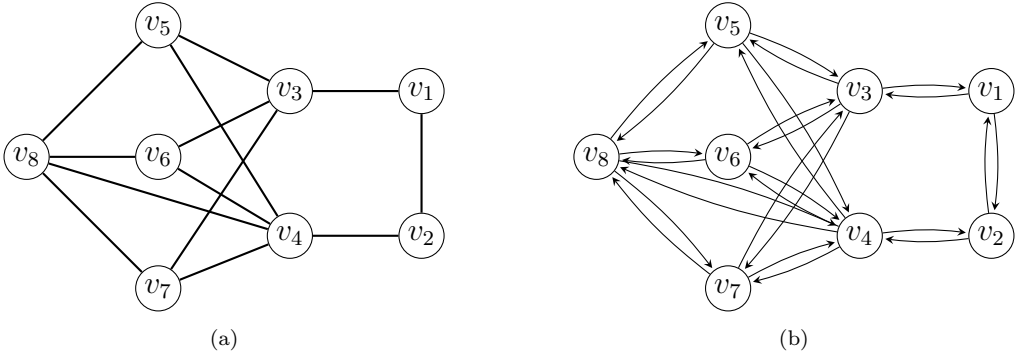
\begin{figure}[tb]
    \centering
    \begin{subfigure}{0.45\textwidth}
    \centering
    \begin{tikzpicture}[scale=0.35, auto,swap, vertex/.style={circle,draw,minimum size=6mm,inner sep=0}]
    \foreach \pos/\name/\label in {{(0,0)/v_6/}, {(0,5)/v_5/}, {(0,-5)/v_7/}, {(5,2.5)/v_3/}, {(10,2.5)/v_1/}, {(5,-3)/v_4/}, {(10,-3)/v_2/}, {(-5,0)/v_8/}}
        \node[vertex, label={\footnotesize $\label$}] (\name) at \pos {$\name$};
    \foreach \source/ \dest in {{v_8/v_5}, {v_8/v_7}, {v_8/v_6}, {v_7/v_3}, {v_7/v_4}, {v_6/v_3}, {v_6/v_4}, {v_5/v_3}, {v_5/v_4}, {v_4/v_2}, {v_3/v_1}, {v_2/v_1}, {v_8/v_4}}
        \path[edge] (\source) -- (\dest);
\end{tikzpicture}
\caption{}
\end{subfigure}
\hspace{0.5cm}
\begin{subfigure}{0.45\textwidth}
\centering
\begin{tikzpicture}[scale=0.35, 
                    auto,
                    swap,
                    > = stealth,
                    shorten > = 1pt,
                    vertex/.style={circle,draw,minimum size=6mm,inner sep=0}]
    \foreach \pos/\name/\label in {{(0,0)/v_6/}, {(0,5)/v_5/}, {(0,-5)/v_7/}, {(5,2.5)/v_3/}, {(10,2.5)/v_1/}, {(5,-3)/v_4/}, {(10,-3)/v_2/}, {(-5,0)/v_8/}}
        \node[vertex, label={\footnotesize $\label$}] (\name) at \pos {$\name$};
    \foreach \source/ \dest in {{v_8/v_5}, {v_8/v_7}, {v_8/v_6}, {v_7/v_3}, {v_7/v_4}, {v_6/v_3}, {v_6/v_4}, {v_5/v_3}, {v_5/v_4}, {v_4/v_2}, {v_3/v_1}, {v_2/v_1}, {v_8/v_4}}
        \path[->] (\source) edge [bend left=7]  (\dest);
    \foreach \source/ \dest in {{v_8/v_5}, {v_8/v_7}, {v_8/v_6}, {v_7/v_3}, {v_7/v_4}, {v_6/v_3}, {v_6/v_4}, {v_5/v_3}, {v_5/v_4}, {v_4/v_2}, {v_3/v_1}, {v_2/v_1}, {v_8/v_4}}
        \path[->] (\dest) edge [bend left=7]  (\source);
\end{tikzpicture}\caption{}
\end{subfigure}
    \caption{Example of transformation of (a) undirected Lemke graph $G = (V,E)$ to (b) bidirected graph $G' = (V,A)$ for use in the MILP formulation. Each undirected edge $(i,j) \in E$ has two corresponding directed edges, also refered to as arcs, $(i,j)$ and $(j,i)$ in arc set $A$.}
    \label{fig:bidirect}
\end{figure}
\subsection{Decision Variables}
\indent We begin by specifying an arbitrary root vertex $r\in V(G)$ and transforming $G$ to a bidirected graph $G' = (V,A)$ to capture child-parent vertex relationships. Specifically, for all undirected edges $(i,j)\in E(G)$, add directed edges $(i,j)$ and $(j,i)$ to the arc set $A(G')$. The relationship between undirected graph $G$ and bidirected graph $G'$ is shown in Figure \ref{fig:bidirect}. Hence, if the original, undirected graph $G$ has $|E(G)| = k$ edges, the bidirected graph $G'$ has $|A(G')| = 2k = m$ directed edges. For the remainder of the paper, we refer to directed edges as arcs and denote them $(i,j) \in A$, where $i\in V(G')$ and $j \in V(G')$ are source and destination vertices, respectively. Next, define a set of integer decision variables $\mathcal{X} \in \{0,1\}^{m\times T}$ such that $X^t\in \{0,1\}^{m}$ represents the arcs of $G'$ in strategy $t$:
\begin{equation}
    x_{i,j}^t = \begin{cases} 1 \quad \text{ if } (i,j) 
    \in E(t)\\ 0 \quad \text{ if } (i,j) \notin E(t), \quad \end{cases}
\end{equation}
where $t\in [T]$ is a tree strategy with vertex set $V(t)$ and edge set $E(t)$, and $(i,j)\in A$ is an arc in $G'$. Note that $x_{i,j}^t$ and $x_{j,i}^t$ are distinct decision variables that represent arcs $(i,j)$ and $(j,i)$ in opposite directions. Further, define a set of integer decision variables $\mathcal{Y} \in \{0,1\}^{n\times T}$ such that $Y^t \in \{0,1\}^{n}$ represents the vertices of $G$ in strategy $t$:
\begin{equation}
    y_{i}^t = \begin{cases} 1 \quad \text{ if } i \in V(t)\backslash \{r\} \\ 0 \quad \text{ if } i \notin V(t)\backslash \{r\}. \quad \end{cases}
\end{equation}
Lastly, define a set of continuous decision variables $\mathcal{Z} \in \mathbb{R}_{\geq 0}^{n \times T}$ where $Z^t \in \mathbb{R}_{\geq 0}^{n}$ represents the weights on the vertices of $G$ in strategy $t$. We denote $z_i^t$ as the weight of vertex $i\in V$ in strategy $t$, where $0 \leq z_i^t\leq 2^{\ell - 1}$ for all $i\in V(G)$.

\subsection{Linear Optimization Problem}
The goal of the proposed mathematical program is to produce a set of tree strategies that minimizes the average tree weight. The tree strategies will be structured such that they bound pebbling numbers using the Covering Lemma \ref{lem:covering} and the optimization problem (\ref{eq:lin-opt}) outlined in \cite{Hurlbert2010}. Automating the generation of these strategies through the approach defined above will improve the application of the pebbling optimization problems introduced in \cite{Cranston2015} and \cite{Hurlbert2010}, and have the potential to produce useful pebbling bounds on a wide range of graphs. \\ 
\indent  We propose the following Tree Strategy Optimization problem $\mathbf{TS}_{G,r}$ on the bidirected graph $G' = (V,A)$ derived from the undirected graph of interest $G$. The MILP takes in the number of tree strategies $T$ and maximum tree diameter $\ell$ as input, and outputs a set of $T$ strategies. We use the notation $V$ for the vertex of the graph $G$, and $V(t)$ for the vertex set for tree strategy $t \in [T]$. The MILP is formulated as: 
\begin{subequations}
\begin{align}
    \label{eq:tree-obj}
    \text{Minimize: }&  \sum_{t=1}^T \sum_{i\in V(t)} z_i^t \\
    \label{eq:tree-constr1}
    \text{Subject to: } &  \sum_{j\in V(t) : (j,i) \in A} x_{j,i}^t = y_i^t  &&t = 1,\ldots, T \text{ and } \forall i\in V\\
    \label{eq:tree-constr2}
    & \sum_{i\in V : (r,i)\in A} y_i^t \geq 1 &&t = 1,\ldots,T\\
    \label{eq:tree-constr3}
    & \sum_{t=1}^T z_i^t \geq T &&\forall i\in V : i\neq r \\
    \label{eq:tree-constr4}
    & z_i^t-2z_j^t+2^{\ell}(1-x_{i,j}^t)\geq 0 &&t = 1,\ldots,T \text{ and } \forall (i,j)\in A : i,j \neq r\\ 
    \label{eq:tree-constr5}
    & z_i^t \leq 2^{\ell -1}\cdot y_i^t &&t = 1,\ldots,T\text{ and } \forall i\in V\\ 
    & y_r^t = 0 && t=1,\ldots,T \label{eq:root0}\\ 
    & x_{i,j}^t \in \{0,1\} &&t = 1,\ldots,T \text{ and } \forall (i,j)\in A \\
    & y_{i}^t \in \{0,1\} &&t = 1,\ldots,T \text{ and } \forall i\in V\\
    & z_i^t \in \mathbb{R}_{\geq 0} &&t = 1,\ldots,T \text{ and } \forall i\in V \label{eq:end}
\end{align}
\end{subequations}

The constraints of the MILP enforce structural limitations on the set of tree strategies and ensure that the weight functions satisfy Definition \ref{def:tree-strategy}. Specifically, (\ref{eq:tree-constr1}), (\ref{eq:tree-constr2}) and (\ref{eq:tree-constr4}) work together to ensure each tree is acyclic, connected and rooted at $r$. While the root $r$ is in every tree, we set the corresponding decision variables equal to zero with (\ref{eq:root0}) to ensure the root has no incoming edges. Additionally, (\ref{eq:tree-constr5}) places an upper bound of $2^{\ell-1}$ on each vertex weight. Lastly, (\ref{eq:tree-constr3}) fixes the minimum weight on each vertex to the number of tree strategies $T$, which corresponds to constant $K$ in Lemma \ref{lem:covering}. Therefore, by minimizing the total weight $M$ by objective function, $\mathbf{TS}_{G,r}$ generates tree strategies that minimize the pebbling bound $\pi(G, r) \leq 
\left \lfloor \frac{M}{K} \right\rfloor + 1$.

\subsection{Theoretical Discussion}
We would like to prove that the tree strategies generated by $\mathbf{TS}_{G,r}$ produces useful tree strategies that satisfy Definition \ref{def:tree-strategy}. Along these lines, we present proof that the program outputs valid weight functions. If a tree strategy $\mathcal{T}$ satisfies constraints (\ref{eq:tree-constr1})-(\ref{eq:end}), we say the strategy is \textit{feasible}. We will argue that all feasible strategies satisfy Definition \ref{def:tree-strategy}. \\ 
\indent Additionally, in $\mathbf{TS}_{G,r}$, we seek strategies that are useful in computing low pebbling bounds. Given a graph $G$ rooted at $r$, we calculate an upper bound on $\pi(G, r)$ by the Covering Lemma \ref{lem:covering}, defining constants $K$ and $M$ and bounding the pebbling number by $\pi(G,r) \leq \left\lfloor\frac{M}{K}\right\rfloor + 1$. We fix the minimum vertex weight $K$ to be the number of strategies $T$ by constraint (\ref{eq:tree-constr3}) and minimize the total weight $M$ by the objective function. More specifically, let $\hat{z}$ be the objective value produced by a feasible set of $T$ strategies. We then calculate bounds by 
\begin{equation}
    \pi(G,r) \leq \left\lfloor \frac{M}{K}\right\rfloor + 1 = \left\lfloor\frac{\hat{z}}{T}\right\rfloor + 1 .
    \label{eq:opt}
\end{equation} Now, we offer proof that an arbitrary feasible tree strategy $\mathcal{T}$ produced by $\mathbf{TS}_{G,r}$ satisfies Definition \ref{def:tree-strategy} and thereby corresponds to a valid tree strategy. Recall that the tree strategies produced by $\mathbf{TS}_{G,r}$ have directed edges such that for each $(i,j) \in E(t)$, $\text{dist}(r,i) < \text{dist}(r,j)$. In other words, edges are directed \textit{away} from the root vertex $r$. 
\begin{lemma}
If $\mathcal{T}$ is a feasible tree strategy generated by $\mathbf{TS}_{G,r}$, then $\mathcal{T}$ is acyclic, connected and rooted at $r$.
\label{lem:tree-structure}
\end{lemma}
\begin{proof}
Let $\mathcal{T} = (V, E, w)$ be an arbitrary feasible tree strategy generated by $\mathbf{TS}_{G,r}$ with index $t$. First, we will show that $\mathcal{T}$ is acyclic by contradiction. Suppose there exists a cycle in $\mathcal{T}$ with vertices $V(\mathcal{T}) = \{i_1, \ldots, i_n\}$ and edges \begin{equation}E(\mathcal{T}) = \{(i_1, i_2), (i_2, i_3), \ldots, (i_{n-1}, i_n), (i_{n}, i_1)\}.\end{equation} Now $y_{i_1}^t = \cdots = y_{i_n}^t = 1$ since all vertices are in $\mathcal{T}$ and $x_{i_1, i_2}^t = \cdots = x_{i_{n-1},i_n}^t = x_{i_{n}, i_1}^t = 1$ by (\ref{eq:tree-constr1}). Fix $i_1$ and walk forward along the cycle. Constraint (\ref{eq:tree-constr4}) says $z_{i_j} \geq 2z_{i_{j+1}}$ for each $j = 1,\ldots, n-1$ and $z_{i_n} \geq 2z_{i_1}$. This implies $z_{i_1} \geq 2^nz_{i_1}$, a contradiction. Hence, $\mathcal{T}$ is acyclic. \\ 
\indent To show $\mathcal{T}$ is rooted at $r$, fix an arbitrary vertex $i \in V(\mathcal{T})$. By (\ref{eq:tree-constr1}), $i$ must have an incoming edge from some vertex $j$ and $r$ must have no incoming edges since $y_r^t = 0$. Walk backwards from $i$ in the opposite direction of the directed arcs. Since there are a finite number of vertices and $r$ has no incoming edges, the walk eventually reaches a vertex with no incoming edges and stops. Since $r$ is the only vertex with no incoming edges that may have an outgoing edge by (\ref{eq:tree-constr1}) and (\ref{eq:tree-constr2}), the backward walk must stop at $r$. Hence, $\mathcal{T}$ is rooted at $r$. By the same argument, $\mathcal{T}$ is also connected. 
\end{proof}

We have shown that feasible tree strategies generated by $\mathbf{TS}_{G,r}$ satisfy the structural conditions of a subtree given in Definition \ref{def:tree-strategy}. Next, we will prove that feasible tree strategies have valid associated weight functions. 
\begin{lemma}  
If $\mathcal{T}$ is a feasible tree strategy generated by $\mathbf{TS}_{G,r}$, then $\mathcal{T}$ has a valid associated weight functions.
\label{lem:feasible-weight}
\end{lemma}
\begin{proof}
Let $\mathcal{T} = (V, E, w)$ be an arbitrary feasible tree strategy generated by $\mathbf{TS}_{G,r}$ with index $t$. Let $i \in V(\mathcal{T})$ be an arbitrary vertex in the tree strategy $\mathcal{T}$ with parent $i^+ = j \in V(\mathcal{T})$, such that $(j,i) \in E(\mathcal{T})$. The decision variables are $y_i^t = y_j^t = 1$ and $x_{j,i}^t  = 1$. Let $z_i^t= w(i)$ and $z_j^t = w(j)$, where $0\leq z_i^t \leq 2^{\ell-1}$ and $0\leq z_j^t \leq 2^{\ell-1}$ by (\ref{eq:tree-constr5}). Then, constraint (\ref{eq:tree-constr4}) says  
\begin{equation}
    z_j^t-2z_i^t+2^{\ell}(1-x_{j,i}^t)= z_j^t-2z_i^t \geq 0
\end{equation}
which implies $ z_j^t \geq 2z_i^t$.
Therefore, $w(i) \geq 2w(i^+)$ for all $i\in V(\mathcal{T})$. \\ 
\indent Similarly, we must show that $w(i) = 0$ for all $i\notin V(\mathcal{T})$. Let $i \notin V(\mathcal{T})$ be an arbitrary vertex not in the tree strategy $\mathcal{T}$. By definition, $y_i^t = 0$ since $i\notin V(\mathcal{T})$. Furthermore, $x_{j,i}^t = 0$ and $x_{i,j}^t = 0$ for all $j \in V(G)$ by (\ref{eq:tree-constr1}). By constraint (\ref{eq:tree-constr5}), we have $z_i^t \leq 2^{\ell - 1} y_i^t \Rightarrow z_i^t \leq 0$. Hence, by nonnegativity, $z_i^t = 0$, as desired. Lastly, we must have that $w(r) = 0$. Let $r\in V(G)$ be the specified root vertex. By definition, $y_r^t = 0$, so the same logic implies that $w(r) = 0$. As such, the tree strategies produced by $\mathbf{TS}_{G,r}$ have valid associated weight functions. 
\end{proof}

Combining the results from Lemma \ref{lem:feasible-weight} and Lemma \ref{lem:tree-structure}, we conclude that all feasible tree strategies generated by the linear program satisfy Definition \ref{def:tree-strategy}. More specifically, feasible strategies have valid associated weight functions and are trees rooted at $r$. Lastly, we consider the usefulness of the tree strategies generated from the optimization problem. Our goal is to calculate tight pebbling bounds by leveraging the objective function to produce feasible strategies that drive down the pebbling bound. If $\mathcal{T}^*$ is an optimal solution to $\mathbf{TS}_{G,r}$, with objective value $z^*$, then $\mathcal{T}^*$ is feasible and has the property that $\hat{z}^* \leq \hat{z}$, where $\hat{z}$ is the objective value of any arbitrary feasible solution. Recall that the objective of $\mathbf{TS}_{G,r}$ is to minimize the total weight $M$ across all strategies. Since constraint (\ref{eq:tree-constr3}) fixes $K=T$, we can apply the Covering Lemma \ref{lem:covering} and use Equation (\ref{eq:opt}) to obtain the pebbling lower bound. Therefore, $\mathbf{TS}_{G,r}$ seeks to generate feasible strategies that minimize pebbling bounds.

\subsection{Symmetric Tree Strategy Optimization Problem}
\label{subsec:tree-sym-opt}
The MILP approach to tree strategy generation allows for the exploration of large sets of tree strategies, far more than could reasonably be generated by hand. Increasing the number of tree strategies generated by $\mathbf{TS}_{G,r}$, determined by input parameter $T$, creates more constraints, which increases the computational complexity of the problem. To enable the optimization solver to produce sizeable sets of tree strategies for large graphs efficiently, we introduce a modification to $\mathbf{TS}_{G,r}$ that utilizes the symmetry of Cartesian product graphs, namely the Lemke square $L\Hsquare L$. \\ 
\indent Let $G = (V,E)$ be a graph rooted at $r$ and $G\Hsquare G$ be its Cartesian product defined in Section 2. Consider a set of $T$ tree strategies generated by $\mathbf{TS}_{G\Hsquare G,r}$. By Lemma \ref{lem:weight-func} and Lemma \ref{lem:tree-structure}, each tree strategy is a valid weight function that satisfies Definition \ref{def:tree-strategy}. We introduce the following definition of a symmetric tree strategy to reduce the number of MILP constraints.
\begin{definition}[Symmetric Tree Strategy]
Let $\mathcal{T} = (V, E, w)$ be an arbitrary, valid tree strategy generated by $\mathbf{TS}_{G\Hsquare G,r}$ for some Cartesian product $G\Hsquare G$ with root $r = (r_1, r_2)$. A tree strategy $\mathcal{T}' = (V', E', w')$ with root $r' = (r_2, r_1)$ is said to be symmetric to $\mathcal{T}$ if and only if for all $(v,u) \in V\backslash \{r\}$ and all $((v,u), (x,y))\in E$, we have $(u,v)\in V'\backslash \{r\}$ and $((u,v), (y,x))\in E'$. Furthermore, for all $(v,u)\in V$, $w(v,u) = w'(u,v)$.  
\label{def:sym-tree}
\end{definition}
    
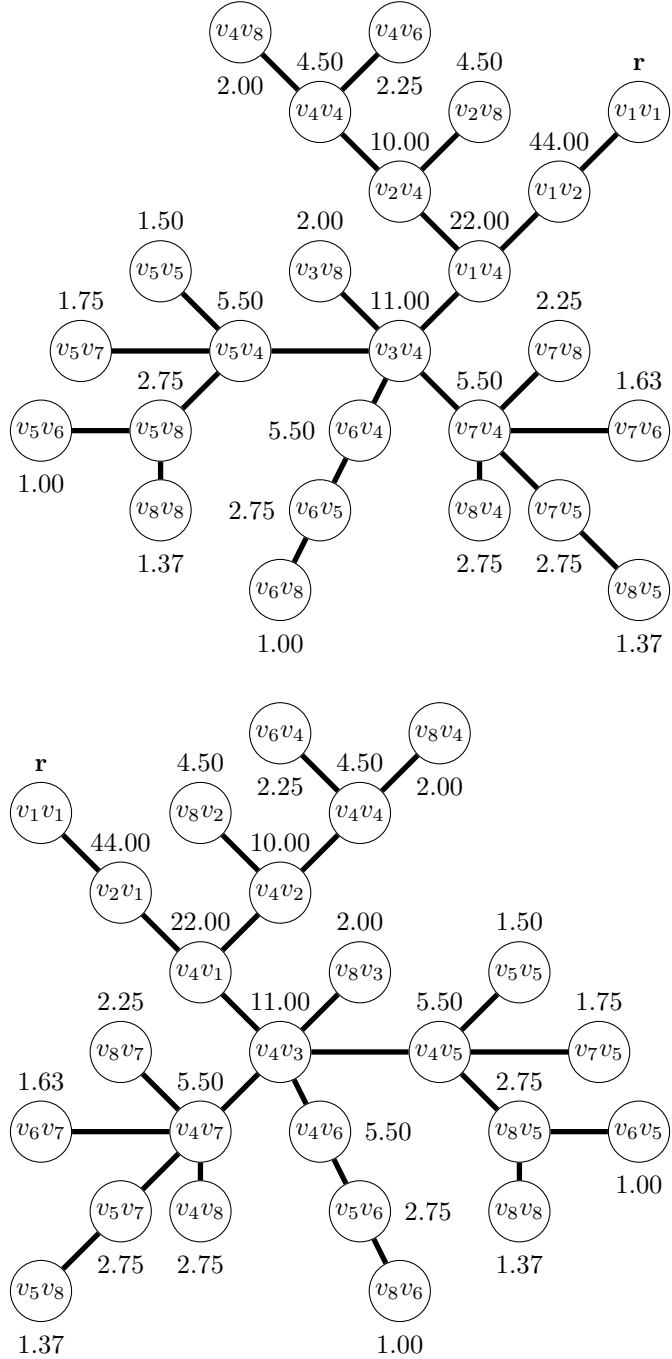
\begin{figure}
\centering
  \begin{tikzpicture}[scale=0.53, auto,swap, vertex/.style={circle,draw,minimum size=7mm,inner sep=0.5mm}]
   
    \foreach \pos/\name/\label in {{(10,4)/v_1v_1/$\textbf{r}$}, {(8,2)/v_1v_2/44.00}, {(6,0)/v_1v_4/22.00}, {(4,2)/v_2v_4/10.00}, {(4,-2)/v_3v_4/11.00},{(2,4)/v_4v_4/4.50},  {(6,4)/v_2v_8/4.50}, {(2,0)/v_3v_8/2.00}, {(0,-2)/v_5v_4/5.50}, {(-2,0)/v_5v_5/1.50}, {(-4,-2)/v_5v_7/1.75}, {(-2,-4)/v_5v_8/2.75},   {(6,-4)/v_7v_4/5.50}, {((8,-2)/v_7v_8/2.25}, {(10,-4)/v_7v_6/1.63}}
        \node[vertex, label={\footnotesize $\label$}] (\name) at \pos {\footnotesize $\name$};

    \foreach \pos/\name/\label in {{(3,-4)/v_6v_4/5.50},{(2,-6)/v_6v_5/2.75}}
         \node[vertex, label={[label distance=]left:{\footnotesize $\label$}}] (\name) at \pos {\footnotesize $\name$};
    \foreach \pos/\name/\label in { {(6,-6)/v_8v_4/2.75}, {(8,-6)/v_7v_5/2.75},{(10,-8)/v_8v_5/1.37},{(-5,-4)/v_5v_6/1.00},{(-2,-6)/v_8v_8/1.37},{(1,-8)/v_6v_8/1.00},{(0,6)/v_4v_8/2.00},{(4,6)/v_4v_6/2.25}}
         \node[vertex, label={[label distance=]below:{\footnotesize $\label$}}] (\name) at \pos {\footnotesize $\name$};

    \foreach \source/ \dest in {{v_1v_1/v_1v_2},{v_1v_2/v_1v_4}, {v_1v_4/v_2v_4}, {v_2v_4/v_4v_4}, {v_4v_4/v_4v_6}, {v_4v_4/v_4v_8},{v_2v_4/v_2v_8}, {v_1v_4/v_3v_4}, {v_3v_4/v_5v_4}, {v_3v_4/v_3v_8}, {v_5v_4/v_5v_5},{v_5v_4/v_5v_7},{v_5v_4/v_5v_8},{v_5v_8/v_5v_6},{v_5v_8/v_8v_8},{v_3v_4/v_6v_4}, {v_6v_4/v_6v_5}, {v_6v_5/v_6v_8},{v_3v_4/v_7v_4},{v_7v_4/v_8v_4},{v_7v_4/v_7v_8},{v_7v_6/v_7v_4},{v_8v_5/v_7v_5}, {v_7v_4/v_7v_5}}
        \path[edge, line width=2pt]  (\source) -- (\dest);
\end{tikzpicture}

\vspace{0.5cm}
  \begin{tikzpicture}[scale=0.53, auto,swap, vertex/.style={circle,draw,minimum size=7mm,inner sep=0.5mm}]
   
    \foreach \pos/\name/\label in {{(-10,4)/v_1v_1/$\textbf{r}$}, {(-8,2)/v_2v_1/44.00}, {(-6,0)/v_4v_1/22.00}, {(-4,2)/v_4v_2/10.00}, {(-4,-2)/v_4v_3/11.00},{(-2,4)/v_4v_4/4.50},   {(-6,4)/v_8v_2/4.50}, {(-2,0)/v_8v_3/2.00}, {(0,-2)/v_4v_5/5.50}, {(2,0)/v_5v_5/1.50}, {(4,-2)/v_7v_5/1.75}, {(2,-4)/v_8v_5/2.75},   {(-6,-4)/v_4v_7/5.50}, {((-8,-2)/v_8v_7/2.25}, {(-10,-4)/v_6v_7/1.63}}
        \node[vertex, label={\footnotesize $\label$}] (\name) at \pos {\footnotesize $\name$};

    \foreach \pos/\name/\label in {{(-3,-4)/v_4v_6/5.50},{(-2,-6)/v_5v_6/2.75}}
         \node[vertex, label={[label distance=]right:{\footnotesize $\label$}}] (\name) at \pos {\footnotesize $\name$};
    \foreach \pos/\name/\label in { {(-6,-6)/v_4v_8/2.75}, {(-8,-6)/v_5v_7/2.75},{(-10,-8)/v_5v_8/1.37},{(5,-4)/v_6v_5/1.00},{(2,-6)/v_8v_8/1.37},{(-1,-8)/v_8v_6/1.00},{(0,6)/v_8v_4/2.00},{(-4,6)/v_6v_4/2.25}}
         \node[vertex, label={[label distance=]below:{\footnotesize $\label$}}] (\name) at \pos {\footnotesize $\name$};

    \foreach \source/ \dest in {{v_1v_1/v_2v_1},{v_2v_1/v_4v_1}, {v_4v_1/v_4v_2}, {v_4v_2/v_4v_4}, {v_4v_4/v_6v_4}, {v_4v_4/v_8v_4},{v_4v_2/v_8v_2}, {v_4v_1/v_4v_3}, {v_4v_3/v_4v_5}, {v_4v_3/v_8v_3}, {v_4v_5/v_5v_5},{v_4v_5/v_7v_5},{v_4v_5/v_8v_5},{v_8v_5/v_6v_5},{v_8v_5/v_8v_8},{v_4v_3/v_4v_6}, {v_4v_6/v_5v_6}, {v_5v_6/v_8v_6},{v_4v_3/v_4v_7},{v_4v_7/v_4v_8},{v_4v_7/v_8v_7},{v_6v_7/v_4v_7},{v_5v_8/v_5v_7}, {v_4v_7/v_5v_7}}
        \path[edge, line width=2pt]  (\source) -- (\dest);
\end{tikzpicture}
 \caption{Example of two symmetric tree strategies generated by $\mathbf{STS}_{G,r}$ on $L\Hsquare L$ with root $r = (v_1, v_1)$.}
 \label{fig:sym-ex}
\end{figure}

Suppose we wish to generate $T$ tree strategies for a graph $G\Hsquare G$, where $T$ is even. Utilizing symmetric pairs of trees, we may generate $T/ 2$ strategies with $\mathbf{TS}_{G\Hsquare G,r}$ and find the corresponding symmetric pair of the trees in this set to generate $T$ total strategies. Evidently, this allows us to generate large amounts of tree strategies at roughly half the computational cost. An example of a symmetric tree strategy pair on $L\Hsquare L$ is shown in Figure \ref{fig:sym-ex}. Before introducing the modified optimization problem that leverages symmetry, we show that the corresponding symmetric tree strategy $\mathcal{T}'$ for a valid tree strategy $\mathcal{T}$ is also valid. 
\begin{lemma}
    If $\mathcal{T}$ is an arbitrary valid tree strategy generated by $\mathbf{TS}_{G\Hsquare G,r}$, then the corresponding symmetric tree $\mathcal{T}'$ is also valid.  
\end{lemma}
\begin{proof}
Let $\mathcal{T}= (V, E, w)$ be a valid tree strategy of a graph $G\Hsquare G$ rooted at $r= (r_1, r_2)$. Consider the symmetric tree $\mathcal{T}'= (V', E', w')$ that satisfies Definition \ref{def:sym-tree}. To show that $\mathcal{T}'$ is a valid tree strategy, we must show that $\mathcal{T}'$ is acyclic, connected and rooted at $r' = (r_1, r_2)$, and that $w'$ is a valid weight function. First, by Definition \ref{def:sym-tree}, $\mathcal{T}'$ must be rooted at $r' = (r_1, r_2)$.

Next, suppose that there exists a cycle in $\mathcal{T}'$ and denote this cycle \begin{equation}C' = ((v_k, v_n),\ldots, (v_i,v_j), \ldots, (v_k, v_n))\end{equation} such that $(v_i, v_j) \neq r$ and for all $i,j\geq0$. By the definition of a symmetric tree strategy, there must exist a corresponding cycle \begin{equation}C = ((v_n, v_k),\ldots, (v_j,v_i), \ldots, (v_n, v_k))\end{equation}. Thus, we have reached a contradiction since $\mathcal{T}$ is acyclic by Lemma \ref{lem:tree-structure}, so $\mathcal{T}'$ must also be acyclic. Lastly, the definition of symmetric trees preserves connectivity, therefore if $\mathcal{T}$ is connected, so is $\mathcal{T}'$. \\ 
\indent Lastly, the weight function $w'$ of the symmetric tree $\mathcal{T}'$ must be valid. Recall that a weight function is valid if $w(v) \geq 2w(v^+)$, where $v^+$ is the parent node of $v$, such that $(v^+, v)\in E(\mathcal{T})$. Let $(u,v), (x,y) \in V$ be arbitrary vertices in $\mathcal{T}$ such that $(x,y)$ is the parent node of $(u,v)$, so $((u,v), (x,y)) \in E$. Since $\mathcal{T}$ is valid, $w(u,v) \geq 2w(x,y)$. Now, consider the corresponding symmetric nodes $(v,u), (y,x) \in V'$ and the edge $((v,u), (y,x))\in E'$. By Definition \ref{def:sym-tree}, $w'(v,u) = w(u,v)$ and $w'(y,x) = w(x,y)$. This implies $w'(v,u) \geq 2w'(y,x)$, as desired. Therefore, the validity of the weight function is preserved through symmetry. 
\end{proof} 
Since all tree strategies produced by $\mathbf{TS}_{G\Hsquare G,r}$ have a corresponding symmetric strategy that satisfies Definition \ref{def:sym-tree}, we may generate a set of $T/2$ strategies and use the symmetric pair of each strategy to produce the desired $T$ strategies while preserving computational cost. To generate these strategies, we define a modified version of $\mathbf{TS}_{G\Hsquare G,r}$ to account for the symmetric strategies. Suppose we have a graph $G\Hsquare G$ rooted at $r$ and we would like to find $T$ tree strategies, where $T$ is even. For each $i=(v,u) \in V(G\Hsquare G)$, define $i'= (u,v)\in V(G\Hsquare G)$. The general framework for the Symmetric Tree Strategy Optimization Problem $\mathbf{STS}_{G\Hsquare G,r}$ is then defined as follows: 
\begin{subequations}
\begin{align}
    \label{eq:tree-obj-sym}
    \text{Minimize: }&  \sum_{t=1}^{T/2} \sum_{i\in V(t)} z_i^t + z_{i'}^t \\\label{eq:tree-constr1-sym}
    \text{Subject to: } &  \sum_{j\in V(t) : (j,i) \in A} x_{j,i}^t = y_i^t  && t=1,\ldots,T/2 \text{ and } \forall i\in V(G)\\\label{eq:tree-constr2-sym}
    & \sum_{i\in V(G) : (r,i)\in A} y_i^t \geq 1 && t=1,\ldots,T/2\\ \label{eq:tree-constr3-sym}
    & \sum_{t=1}^{T/2} z_i^t + z_{i'}^t \geq T && \forall i\in V(G) : i\neq r \\\label{eq:tree-constr4-sym}
    & z_i^t-2z_j^t+2^{\ell}(1-x_{i,j}^t)\geq 0 && t=1,\ldots,T/2 \text{ and } \forall (i,j)\in A(G) : i,j \neq r\\ \label{eq:tree-constr5-sym}
    & z_i^t \leq 2^{\ell -1}\cdot y_i^t && t=1,\ldots,T/2 \text{ and } \forall i\in V(G) \\ 
    & y_r^t = 0 && t=1,\ldots,T/2 \label{eq:root0-sym}\\ 
    & x_{i,j}^t \in \{0,1\} &&t=1,\ldots,T/2 \text{ and }  \forall (i,j)\in A(G) \\
    & y_{i}^t \in \{0,1\} &&t=1,\ldots,T/2\text{ and }  \forall i\in V(G)\\
    & z_i^t \in \mathbb{R}_{\geq 0} &&t=1,\ldots,T/2 \text{ and }  \forall i\in V(G)
\end{align}
\end{subequations}
We can then generate a complete set of $T$ tree strategies by taking the corresponding symmetric pair of each strategy produced by $\mathbf{STS}_{G\Hsquare G,r}$. Let $\hat{z}$ be the objective value associated with a feasible set of $T$ strategies. As before, the pebbling bound is given by 
\begin{equation}
    \pi(G\Hsquare G, r) \leq \left\lfloor\frac{\hat{z}}{T}\right\rfloor +1.
\end{equation}
Constraints (\ref{eq:tree-constr1-sym})-(\ref{eq:tree-constr2-sym}) and (\ref{eq:tree-constr4-sym})-(\ref{eq:root0-sym}) those in   $\textbf{TS}_{G\Hsquare G,r}$ to preserve the validity of the tree strategy weight functions. Aside from reducing the number of tree strategies by a factor of 1/2, and modifying the objective function to reflect the symmetric properties of the weight functions generated, we modify constraint (\ref{eq:tree-constr3-sym}). Recall that constraint (\ref{eq:tree-constr3-sym}) serves to encourage diverse tree strategies by ensuring that each vertex weight is bounded below by the total number of strategies produced. In the case of the symmetric strategies, if we produce $T/2$ symmetric strategies, the total number of strategies used to prove the pebbling bound is $T$. Further, by symmetry $z_i^t = z_{i'}^t$, with vertex $i$ appearing in $\mathcal{T}$ and $i'$ appearing in its symmetric counterpart $\mathcal{T}'$. Thus, to ensure $i$ and $i'$ each has weight at least $T/2$ in $\mathcal{T}$ and symmetric strategy $\mathcal{T}'$, respectively, we impose constraint (\ref{eq:tree-constr3-sym}). \\ 
\indent This approach drastically reduces the number of constraints in the MILP. Constraints (\ref{eq:tree-constr1-sym})-(\ref{eq:tree-constr2-sym}) are structural constraints that ensure solutions are trees rooted at $r$, since we are now placing this constraint on $T/2$ trees, we have $T/2$ less constraints than in $\mathbf{TS}_{G\Hsquare G,r}$. Constraints (\ref{eq:tree-constr4-sym}) and (\ref{eq:tree-constr5-sym}) enforce and bound valid weight functions in a given tree. We have $T/2\cdot |V|$ and $T\cdot |E|$ instances of constraints (\ref{eq:tree-constr4-sym}) and (\ref{eq:tree-constr5-sym}), respectively, so we reduce each of these constraints by a factor of $1/2$. Constraints (\ref{eq:tree-constr1-sym})-(\ref{eq:tree-constr2-sym}) and (\ref{eq:tree-constr4-sym})-(\ref{eq:tree-constr5-sym}) enforce Definition \ref{def:tree-strategy}, so they remain unchanged to ensure the tree strategies generated are valid. Lastly, we modify objective function (\ref{eq:tree-obj-sym}) and constraint (\ref{eq:tree-constr3-sym}) accordingly to incorporate the symmetric approach.

\section{Pebbling Results}
\label{sec:results}
The MILP framework above gives many encouraging and impressive results using non-basic tree strategies. As previously mentioned, the code used to implement the optimization problems is available as an open source codebase for further experimentation.\footnote{\url{https://github.com/dominicflocco/Graph\_Pebbling}} While we present notable results in the Appendix; the \texttt{.lp} and \texttt{.log} files, along with certificates and strategy visualizations for all experiments, can also be found in the repository. For the results presented below, we use the popular \texttt{Gurobi}\footnote{\url{https://www.gurobi.com}} mathematical optimization software, yet the code base also implements the \texttt{IBM ILOG CPLEX}\footnote{\url{https://www.ibm.com/analytics/cplex-optimizer}} solver. Neither solver showed significant improvement over the other in experimentation. For each solver, we modify the MILP Focus to find solutions quickly. Experiments were run on a system with an Intel Xeon Processor E5-2620 v4 with 16 cores, each running at 2.1GHz. The system has 128GB of memory and two NUMA nodes.
\subsection{Bound Improvements}
We apply the Tree Strategy Optimization problem to larger graphs in an effort to improve bounds found using manually generated tree strategies in Hurlbert \cite{Hurlbert2010}. Applying linear optimization techniques to automate the generation of tree strategies is particularly useful for large graphs because devising useful strategies for such graphs is often difficult and tedious. Furthermore, the challenge of generating these strategies manually only increases as the number of strategies generated grows; and, large sets of tree strategies may be useful in improve pebbling bounds. The two graphs of interest for this section are the 4\textsuperscript{th} weak Bruhat $B_4$, shown in Figure \ref{fig:bruhat-template2}, and the Lemke square graph $L \Hsquare L$. \\ 
\indent The results on the 4\textsuperscript{th} weak Bruhat graph $B_4$ exhibits that improved pebbling bounds using tree strategies are possible as $T$ is increased. Hurlbert \cite{Hurlbert2010} presents a certificate and visualization that shows $\pi(B_4,r) \leq 72$ using 3 strategies. However, using $\textbf{TS}_{B_4,r}$ with larger $T$ and $\ell$, we obtain the following result.
\begin{figure}
\centering
\begin{tikzpicture}[scale=0.5, auto,swap, vertex/.style={circle,draw,minimum size=6mm,inner sep=0}]
    \foreach \pos/\name/\label in {{(-5,5)/v_1/$r$}, {(5,5)/v_2/}, {(5,-5)/v_4/}, {(-5,-5)/v_3/}, {(-4,4)/v_5/}, {(4,4)/v_8/}, {(-4,-4)/v_{21}/}, {(4,-4)/v_{24}/}, {(-1,-1)/v_{15}/}, {(1,1)/v_{14}/},{(1,-1)/v_{16}/}, {(-1,1)/v_{13}/}, {(-2,4)/v_6/}, {(2,4)/v_7/}, {(-4,2)/v_9/}, {(-4,-2)/v_{17}/}, {(4,2)/v_{12}/}, {(4,-2)/v_{20}/}, {(-2,-4)/v_{22}/}, {(2,-4)/v_{23}/}, {(2,2)/v_{11}/}, {(2,-2)/v_{19}/}, {(-2,-2)/v_{18}/}, {(-2,2)/v_{10}/}}
        \node[vertex, label={\small $\label$}] (\name) at \pos {\small $\name$};
    \foreach \source/ \dest in {{v_1/v_2}, {v_2/v_4}, {v_4/v_3}, {v_3/v_1}, {v_3/v_{21}}, {v_4/v_{24}}, {v_2/v_8}, {v_1/v_5}, {v_5/v_6}, {v_6/v_7}, {v_7/v_8}, {v_8/v_{12}}, {v_{12}/v_{20}}, {v_{20}/v_{24}}, {v_{24}/v_{23}}, {v_{23}/v_{22}}, {v_{22}/v_{21}}, {v_{21}/v_{17}}, {v_{17}/v_9}, {v_9/v_5}, {v_9/v_{10}}, {v_6/v_{10}}, {v_{10}/v_{13}}, {v_7/v_{11}}, {v_{11}/v_{12}}, {v_{23}/v_{19}}, {v_{19}/v_{20}}, {v_{18}/v_{17}},{v_{22}/v_{18}},{v_{18}/v_{15}}, {v_{19}/v_{16}}, {v_{11}/v_{14}}, {v_{13}/v_{15}}, {v_{13}/v_{14}}, {v_{14}/v_{16}}, {v_{16}/v_{15}}}
        \path[edge] (\source) -- (\dest);

\end{tikzpicture}

\caption{4\textsuperscript{th} weak Bruhat Graph $B_4$ on $n=24$ vertices with $r = v_1$.}
\label{fig:bruhat-template2}
\end{figure}
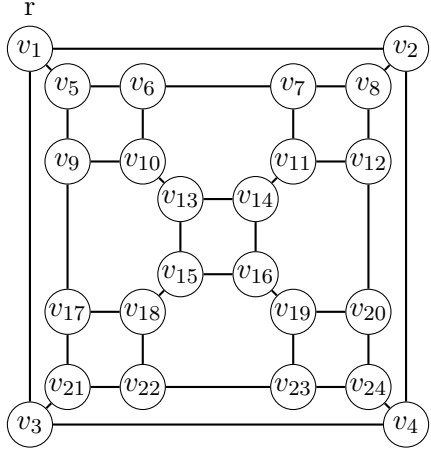
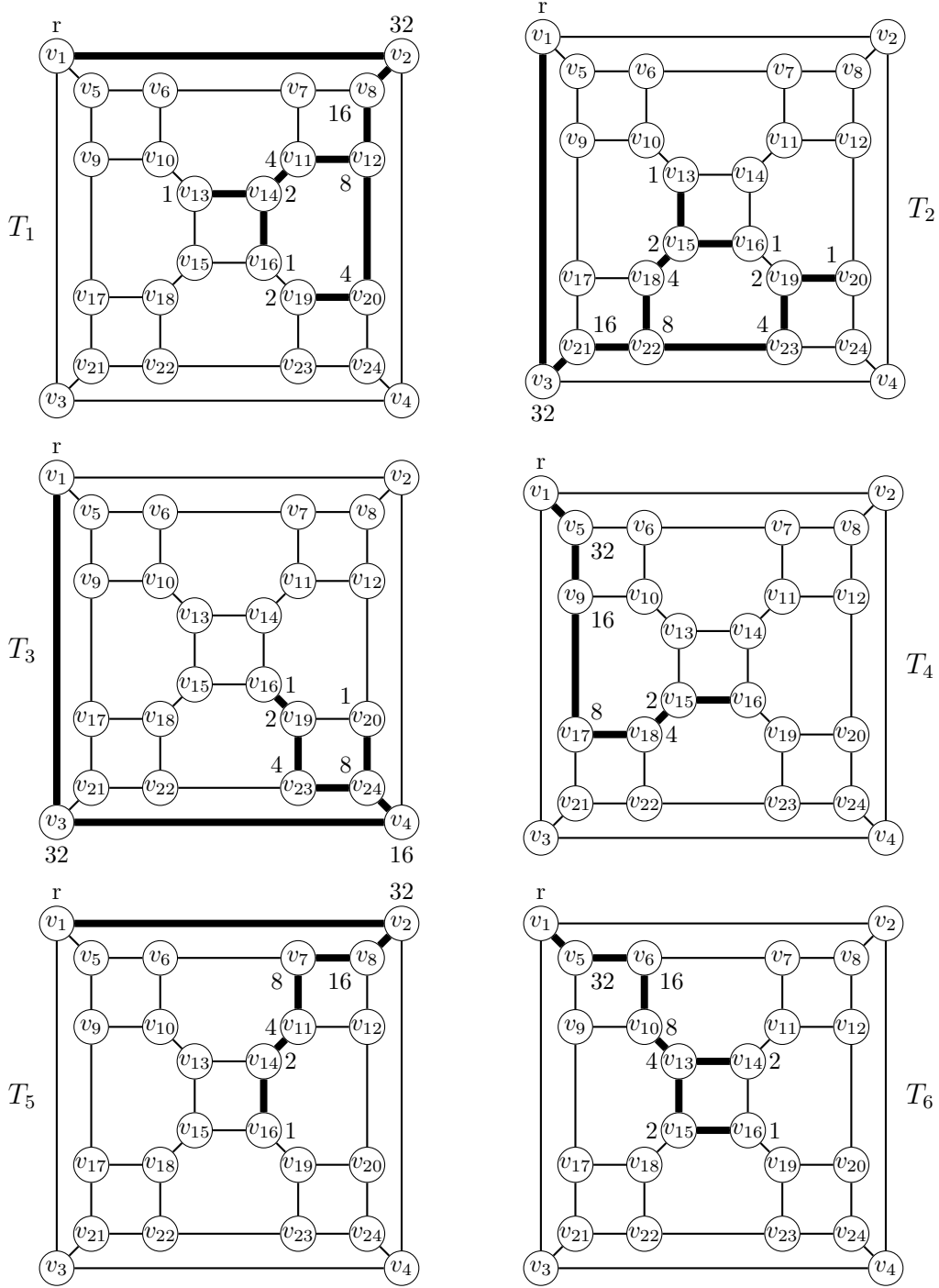
\begin{figure}
\centering
\begin{subfigure}{0.45\textwidth}
\centering
    \begin{tikzpicture}[scale=0.5, auto,swap, vertex/.style={circle,draw,minimum size=5mm,inner sep=0}]
    
    \foreach \pos/\name/\label in 
        {{(-5,5)/v_1/$r$}, {(5,5)/v_2/32}}
        \node[vertex, label={[label distance=-0.5mm]above:{\footnotesize $\label$}}] (\name) at \pos {\footnotesize $\name$};
    \foreach \pos/\name/\label in {{(4,4)/v_8/16},{(4,2)/v_{12}/8}, {(2,4)/v_7/}}
         \node[vertex, label={[label distance=-1.5mm]225:{\footnotesize $\label$}}] (\name) at \pos {\footnotesize $\name$};
    \foreach \pos/\name/\label in {{(-1,1)/v_{13}/1}, {(2,2)/v_{11}/4},{(-1,-1)/v_{15}/},{(2,-2)/v_{19}/2}}
         \node[vertex, label={[label distance=-1mm]left:{\footnotesize $\label$}}] (\name) at \pos {\footnotesize $\name$};
    \foreach \pos/\name/\label in {{(-4,-2)/v_{17}/},{(-4,-4)/v_{21}/},{(-2,-4)/v_{22}/}}
         \node[vertex, label={[label distance=-1.5mm]45:{\footnotesize $\label$}}] (\name) at \pos {\footnotesize $\name$};
    \foreach \pos/\name/\label in {{(1,1)/v_{14}/2},{(-2,2)/v_{10}/},{(-2,-2)/v_{18}/},{(1,-1)/v_{16}/1}}
         \node[vertex, label={[label distance=-1mm]right:{\footnotesize $\label$}}] (\name) at \pos {\footnotesize $\name$};
    \foreach \pos/\name/\label in {{(4,-2)/v_{20}/4},{(2,-4)/v_{23}/},{(4,-4)/v_{24}/}}
         \node[vertex, label={[label distance=-1.5mm]135:{\footnotesize $\label$}}] (\name) at \pos {\footnotesize $\name$};
    \foreach \pos/\name/\label in {{(-2,4)/v_6/},{(-4,2)/v_9/},{(-4,4)/v_5/}}
         \node[vertex, label={[label distance=-1.5mm]315:{\footnotesize $\label$}}] (\name) at \pos {\footnotesize $\name$};
    \foreach \pos/\name/\label in { {(-5,-5)/v_3/},{(5,-5)/v_4/}}
         \node[vertex, label={[label distance=-1mm]below:{\footnotesize $\label$}}] (\name) at \pos {\footnotesize $\name$};
    
    \foreach \source/ \dest in {{v_1/v_2}, {v_2/v_4}, {v_4/v_3}, {v_3/v_1}, {v_3/v_{21}}, {v_4/v_{24}}, {v_2/v_8}, {v_1/v_5}, {v_5/v_6}, {v_6/v_7}, {v_7/v_8}, {v_8/v_{12}}, {v_{12}/v_{20}}, {v_{20}/v_{24}}, {v_{24}/v_{23}}, {v_{23}/v_{22}}, {v_{22}/v_{21}}, {v_{21}/v_{17}}, {v_{17}/v_9}, {v_9/v_5}, {v_9/v_{10}}, {v_6/v_{10}}, {v_{10}/v_{13}}, {v_7/v_{11}}, {v_{11}/v_{12}}, {v_{23}/v_{19}}, {v_{19}/v_{20}}, {v_{18}/v_{17}},{v_{22}/v_{18}},{v_{18}/v_{15}}, {v_{19}/v_{16}}, {v_{11}/v_{14}}, {v_{13}/v_{15}}, {v_{13}/v_{14}}, {v_{14}/v_{16}}, {v_{16}/v_{15}}}
        \path[edge] (\source) -- (\dest);

    \foreach \source/\dest in {{v_1/v_2}, {v_2/v_8}, {v_8/v_{12}}, {v_{12}/v_{20}}, {v_{20}/v_{19}}, {v_{12}/v_{11}}, {v_{11}/v_{14}}, {v_{14}/v_{16}}, {v_{14}/v_{13}}}
        \path[edge, line width=3pt] (\source) -- (\dest);
    \node[] at (-6,0) {$T_1$};
\end{tikzpicture}
\end{subfigure}
\hspace{5mm}
\begin{subfigure}{0.45\textwidth}
\centering
        \begin{tikzpicture}[scale=0.5, auto,swap, vertex/.style={circle,draw,minimum size=5mm,inner sep=0}]
    \foreach \pos/\name/\label in 
        {{(-5,5)/v_1/$r$}, {(5,5)/v_2/}}
        \node[vertex, label={[label distance=-0.5mm]above:{\footnotesize $\label$}}] (\name) at \pos {\footnotesize $\name$};
    \foreach \pos/\name/\label in {{(4,4)/v_8/},{(4,2)/v_{12}/}, {(2,4)/v_7/}}
         \node[vertex, label={[label distance=-1.5mm]225:{\footnotesize $\label$}}] (\name) at \pos {\footnotesize $\name$};
    \foreach \pos/\name/\label in {{(-1,1)/v_{13}/1}, {(2,2)/v_{11}/},{(-1,-1)/v_{15}/2},{(2,-2)/v_{19}/2}}
         \node[vertex, label={[label distance=-1mm]left:{\footnotesize $\label$}}] (\name) at \pos {\footnotesize $\name$};
    \foreach \pos/\name/\label in {{(-4,-2)/v_{17}/},{(-4,-4)/v_{21}/16},{(-2,-4)/v_{22}/8}}
         \node[vertex, label={[label distance=-1.5mm]45:{\footnotesize $\label$}}] (\name) at \pos {\footnotesize $\name$};
    \foreach \pos/\name/\label in {{(1,1)/v_{14}/},{(-2,2)/v_{10}/},{(-2,-2)/v_{18}/4},{(1,-1)/v_{16}/1}}
         \node[vertex, label={[label distance=-1mm]right:{\footnotesize $\label$}}] (\name) at \pos {\footnotesize $\name$};
    \foreach \pos/\name/\label in {{(4,-2)/v_{20}/1},{(2,-4)/v_{23}/4},{(4,-4)/v_{24}/}}
         \node[vertex, label={[label distance=-1.5mm]135:{\footnotesize $\label$}}] (\name) at \pos {\footnotesize $\name$};
    \foreach \pos/\name/\label in {{(-2,4)/v_6/},{(-4,2)/v_9/},{(-4,4)/v_5/}}
         \node[vertex, label={[label distance=-1.5mm]315:{\footnotesize $\label$}}] (\name) at \pos {\footnotesize $\name$};
    \foreach \pos/\name/\label in { {(-5,-5)/v_3/32},{(5,-5)/v_4/}}
         \node[vertex, label={[label distance=-0.5mm]below:{\footnotesize $\label$}}] (\name) at \pos {\footnotesize $\name$};
    \foreach \source/ \dest in {{v_1/v_2}, {v_2/v_4}, {v_4/v_3}, {v_3/v_1}, {v_3/v_{21}}, {v_4/v_{24}}, {v_2/v_8}, {v_1/v_5}, {v_5/v_6}, {v_6/v_7}, {v_7/v_8}, {v_8/v_{12}}, {v_{12}/v_{20}}, {v_{20}/v_{24}}, {v_{24}/v_{23}}, {v_{23}/v_{22}}, {v_{22}/v_{21}}, {v_{21}/v_{17}}, {v_{17}/v_9}, {v_9/v_5}, {v_9/v_{10}}, {v_6/v_{10}}, {v_{10}/v_{13}}, {v_7/v_{11}}, {v_{11}/v_{12}}, {v_{23}/v_{19}}, {v_{19}/v_{20}}, {v_{18}/v_{17}},{v_{22}/v_{18}},{v_{18}/v_{15}}, {v_{19}/v_{16}}, {v_{11}/v_{14}}, {v_{13}/v_{15}}, {v_{13}/v_{14}}, {v_{14}/v_{16}}, {v_{16}/v_{15}}}
        \path[edge] (\source) -- (\dest);
    \foreach \source/\dest in {{v_1/v_3}, {v_3/v_{21}}, {v_{21}/v_{22}},{v_{22}/v_{23}},{v_{23}/v_{19}},{v_{19}/v_{20}},{v_{22}/v_{18}},{v_{18}/v_{15}},{v_{15}/v_{16}},{v_{15}/v_{13}}}
        \path[edge, line width=3pt] (\source) -- (\dest);
        \node[] at (6,0) {$T_2$};
\end{tikzpicture}
\end{subfigure}
\begin{subfigure}{0.45\textwidth}
\centering
        \begin{tikzpicture}[scale=0.5, auto,swap, vertex/.style={circle,draw,minimum size=5mm,inner sep=0}]
    \foreach \pos/\name/\label in 
        {{(-5,5)/v_1/$r$}, {(5,5)/v_2/}}
        \node[vertex, label={[label distance=-0.5mm]above:{\footnotesize $\label$}}] (\name) at \pos {\footnotesize $\name$};
    \foreach \pos/\name/\label in {{(4,4)/v_8/},{(4,2)/v_{12}/}, {(2,4)/v_7/}}
         \node[vertex, label={[label distance=-1.5mm]225:{\footnotesize $\label$}}] (\name) at \pos {\footnotesize $\name$};
    \foreach \pos/\name/\label in {{(-1,1)/v_{13}/}, {(2,2)/v_{11}/},{(-1,-1)/v_{15}/},{(2,-2)/v_{19}/2}}
         \node[vertex, label={[label distance=-1mm]left:{\footnotesize $\label$}}] (\name) at \pos {\footnotesize $\name$};
    \foreach \pos/\name/\label in {{(-4,-2)/v_{17}/},{(-4,-4)/v_{21}/},{(-2,-4)/v_{22}/}}
         \node[vertex, label={[label distance=-1.5mm]45:{\footnotesize $\label$}}] (\name) at \pos {\footnotesize $\name$};
    \foreach \pos/\name/\label in {{(1,1)/v_{14}/},{(-2,2)/v_{10}/},{(-2,-2)/v_{18}/},{(1,-1)/v_{16}/1}}
         \node[vertex, label={[label distance=-1mm]right:{\footnotesize $\label$}}] (\name) at \pos {\footnotesize $\name$};
    \foreach \pos/\name/\label in {{(4,-2)/v_{20}/1},{(2,-4)/v_{23}/4},{(4,-4)/v_{24}/8}}
         \node[vertex, label={[label distance=-1.5mm]135:{\footnotesize $\label$}}] (\name) at \pos {\footnotesize $\name$};
    \foreach \pos/\name/\label in {{(-2,4)/v_6/},{(-4,2)/v_9/},{(-4,4)/v_5/}}
         \node[vertex, label={[label distance=-1.5mm]315:{\footnotesize $\label$}}] (\name) at \pos {\footnotesize $\name$};
    \foreach \pos/\name/\label in { {(-5,-5)/v_3/32},{(5,-5)/v_4/16}}
         \node[vertex, label={[label distance=-0.5mm]below:{\footnotesize $\label$}}] (\name) at \pos {\footnotesize $\name$};
    \foreach \source/ \dest in {{v_1/v_2}, {v_2/v_4}, {v_4/v_3}, {v_3/v_1}, {v_3/v_{21}}, {v_4/v_{24}}, {v_2/v_8}, {v_1/v_5}, {v_5/v_6}, {v_6/v_7}, {v_7/v_8}, {v_8/v_{12}}, {v_{12}/v_{20}}, {v_{20}/v_{24}}, {v_{24}/v_{23}}, {v_{23}/v_{22}}, {v_{22}/v_{21}}, {v_{21}/v_{17}}, {v_{17}/v_9}, {v_9/v_5}, {v_9/v_{10}}, {v_6/v_{10}}, {v_{10}/v_{13}}, {v_7/v_{11}}, {v_{11}/v_{12}}, {v_{23}/v_{19}}, {v_{19}/v_{20}}, {v_{18}/v_{17}},{v_{22}/v_{18}},{v_{18}/v_{15}}, {v_{19}/v_{16}}, {v_{11}/v_{14}}, {v_{13}/v_{15}}, {v_{13}/v_{14}}, {v_{14}/v_{16}}, {v_{16}/v_{15}}}
        \path[edge] (\source) -- (\dest);
    \foreach \source/\dest in {{v_1/v_3},{v_3/v_4},{v_4/v_{24}},{v_{24}/v_{20}},{v_{24}/v_{23}},{v_{23}/v_{19}},{v_{19}/v_{16}}}
        \path[edge, line width=3pt] (\source) -- (\dest);
    \node[] at (-6,0) {$T_3$};
\end{tikzpicture}
\end{subfigure}
\hspace{5mm}
\begin{subfigure}{0.45\textwidth}
\centering
       \begin{tikzpicture}[scale=0.5, auto,swap, vertex/.style={circle,draw,minimum size=5mm,inner sep=0}]
    \foreach \pos/\name/\label in 
        {{(-5,5)/v_1/$r$}, {(5,5)/v_2/}}
        \node[vertex, label={[label distance=-0.5mm]above:{\footnotesize $\label$}}] (\name) at \pos {\footnotesize $\name$};
    \foreach \pos/\name/\label in {{(4,4)/v_8/},{(4,2)/v_{12}/}, {(2,4)/v_7/}}
         \node[vertex, label={[label distance=-1.5mm]225:{\footnotesize $\label$}}] (\name) at \pos {\footnotesize $\name$};
    \foreach \pos/\name/\label in {{(-1,1)/v_{13}/}, {(2,2)/v_{11}/},{(-1,-1)/v_{15}/2},{(2,-2)/v_{19}/}}
         \node[vertex, label={[label distance=-1mm]left:{\footnotesize $\label$}}] (\name) at \pos {\footnotesize $\name$};
    \foreach \pos/\name/\label in {{(-4,-2)/v_{17}/8},{(-4,-4)/v_{21}/},{(-2,-4)/v_{22}/}}
         \node[vertex, label={[label distance=-1.5mm]45:{\footnotesize $\label$}}] (\name) at \pos {\footnotesize $\name$};
    \foreach \pos/\name/\label in {{(1,1)/v_{14}/},{(-2,2)/v_{10}/},{(-2,-2)/v_{18}/4},{(1,-1)/v_{16}/}}
         \node[vertex, label={[label distance=-1mm]right:{\footnotesize $\label$}}] (\name) at \pos {\footnotesize $\name$};
    \foreach \pos/\name/\label in {{(4,-2)/v_{20}/},{(2,-4)/v_{23}/},{(4,-4)/v_{24}/}}
         \node[vertex, label={[label distance=-1.5mm]135:{\footnotesize $\label$}}] (\name) at \pos {\footnotesize $\name$};
    \foreach \pos/\name/\label in {{(-2,4)/v_6/},{(-4,2)/v_9/16},{(-4,4)/v_5/32}}
         \node[vertex, label={[label distance=-1.5mm]315:{\footnotesize $\label$}}] (\name) at \pos {\footnotesize $\name$};
    \foreach \pos/\name/\label in { {(-5,-5)/v_3/},{(5,-5)/v_4/}}
         \node[vertex, label={[label distance=-0.5mm]below:{\footnotesize $\label$}}] (\name) at \pos {\footnotesize $\name$};
    \foreach \source/ \dest in {{v_1/v_2}, {v_2/v_4}, {v_4/v_3}, {v_3/v_1}, {v_3/v_{21}}, {v_4/v_{24}}, {v_2/v_8}, {v_1/v_5}, {v_5/v_6}, {v_6/v_7}, {v_7/v_8}, {v_8/v_{12}}, {v_{12}/v_{20}}, {v_{20}/v_{24}}, {v_{24}/v_{23}}, {v_{23}/v_{22}}, {v_{22}/v_{21}}, {v_{21}/v_{17}}, {v_{17}/v_9}, {v_9/v_5}, {v_9/v_{10}}, {v_6/v_{10}}, {v_{10}/v_{13}}, {v_7/v_{11}}, {v_{11}/v_{12}}, {v_{23}/v_{19}}, {v_{19}/v_{20}}, {v_{18}/v_{17}},{v_{22}/v_{18}},{v_{18}/v_{15}}, {v_{19}/v_{16}}, {v_{11}/v_{14}}, {v_{13}/v_{15}}, {v_{13}/v_{14}}, {v_{14}/v_{16}}, {v_{16}/v_{15}}}
        \path[edge] (\source) -- (\dest);
    \foreach \source/\dest in {{v_1/v_5},{v_5/v_9},{v_9/v_{17}},{v_{17}/v_{18}},{v_{18}/v_{15}},{v_{15}/v_{16}}}
        \path[edge, line width=3pt] (\source) -- (\dest);
        \node[] at (6,0) {$T_4$};
\end{tikzpicture}
\end{subfigure}
\begin{subfigure}{0.45\textwidth}
\centering
        \begin{tikzpicture}[scale=0.5, auto,swap, vertex/.style={circle,draw,minimum size=5mm,inner sep=0}]
    \foreach \pos/\name/\label in 
        {{(-5,5)/v_1/$r$}, {(5,5)/v_2/32}}
        \node[vertex, label={[label distance=-0.5mm]above:{\footnotesize $\label$}}] (\name) at \pos {\footnotesize $\name$};
    \foreach \pos/\name/\label in {{(4,4)/v_8/16},{(4,2)/v_{12}/}, {(2,4)/v_7/8}}
         \node[vertex, label={[label distance=-1.5mm]225:{\footnotesize $\label$}}] (\name) at \pos {\footnotesize $\name$};
    \foreach \pos/\name/\label in {{(-1,1)/v_{13}/}, {(2,2)/v_{11}/4},{(-1,-1)/v_{15}/},{(2,-2)/v_{19}/}}
         \node[vertex, label={[label distance=-1mm]left:{\footnotesize $\label$}}] (\name) at \pos {\footnotesize $\name$};
    \foreach \pos/\name/\label in {{(-4,-2)/v_{17}/},{(-4,-4)/v_{21}/},{(-2,-4)/v_{22}/}}
         \node[vertex, label={[label distance=-1.5mm]45:{\footnotesize $\label$}}] (\name) at \pos {\footnotesize $\name$};
    \foreach \pos/\name/\label in {{(1,1)/v_{14}/2},{(-2,2)/v_{10}/},{(-2,-2)/v_{18}/},{(1,-1)/v_{16}/1}}
         \node[vertex, label={[label distance=-1mm]right:{\footnotesize $\label$}}] (\name) at \pos {\footnotesize $\name$};
    \foreach \pos/\name/\label in {{(4,-2)/v_{20}/},{(2,-4)/v_{23}/},{(4,-4)/v_{24}/}}
         \node[vertex, label={[label distance=-1.5mm]135:{\footnotesize $\label$}}] (\name) at \pos {\footnotesize $\name$};
    \foreach \pos/\name/\label in {{(-2,4)/v_6/},{(-4,2)/v_9/},{(-4,4)/v_5/}}
         \node[vertex, label={[label distance=-1.5mm]315:{\footnotesize $\label$}}] (\name) at \pos {\footnotesize $\name$};
    \foreach \pos/\name/\label in { {(-5,-5)/v_3/},{(5,-5)/v_4/}}
         \node[vertex, label={[label distance=-0.5mm]below:{\footnotesize $\label$}}] (\name) at \pos {\footnotesize $\name$};
    \foreach \source/ \dest in {{v_1/v_2}, {v_2/v_4}, {v_4/v_3}, {v_3/v_1}, {v_3/v_{21}}, {v_4/v_{24}}, {v_2/v_8}, {v_1/v_5}, {v_5/v_6}, {v_6/v_7}, {v_7/v_8}, {v_8/v_{12}}, {v_{12}/v_{20}}, {v_{20}/v_{24}}, {v_{24}/v_{23}}, {v_{23}/v_{22}}, {v_{22}/v_{21}}, {v_{21}/v_{17}}, {v_{17}/v_9}, {v_9/v_5}, {v_9/v_{10}}, {v_6/v_{10}}, {v_{10}/v_{13}}, {v_7/v_{11}}, {v_{11}/v_{12}}, {v_{23}/v_{19}}, {v_{19}/v_{20}}, {v_{18}/v_{17}},{v_{22}/v_{18}},{v_{18}/v_{15}}, {v_{19}/v_{16}}, {v_{11}/v_{14}}, {v_{13}/v_{15}}, {v_{13}/v_{14}}, {v_{14}/v_{16}}, {v_{16}/v_{15}}}
        \path[edge] (\source) -- (\dest);
    \foreach \source/\dest in {{v_1/v_2},{v_2/v_8},{v_8/v_7},{v_7/v_{11}},{v_{11}/v_{14}},{v_{14}/v_{16}}}
        \path[edge, line width=3pt] (\source) -- (\dest);
    \node[] at (-6,0) {$T_5$};
\end{tikzpicture}
\end{subfigure}
\hspace{5mm}
\begin{subfigure}{0.45\textwidth}
\centering
       \begin{tikzpicture}[scale=0.5, auto,swap, vertex/.style={circle,draw,minimum size=5mm,inner sep=0}]
   \foreach \pos/\name/\label in 
        {{(-5,5)/v_1/$r$}, {(5,5)/v_2/}}
        \node[vertex, label={[label distance=-0.5mm]above:{\footnotesize $\label$}}] (\name) at \pos {\footnotesize $\name$};
    \foreach \pos/\name/\label in {{(4,4)/v_8/},{(4,2)/v_{12}/}, {(2,4)/v_7/}}
         \node[vertex, label={[label distance=-1.5mm]225:{\footnotesize $\label$}}] (\name) at \pos {\footnotesize $\name$};
    \foreach \pos/\name/\label in {{(-1,1)/v_{13}/4}, {(2,2)/v_{11}/},{(-1,-1)/v_{15}/2},{(2,-2)/v_{19}/}}
         \node[vertex, label={[label distance=-1mm]left:{\footnotesize $\label$}}] (\name) at \pos {\footnotesize $\name$};
    \foreach \pos/\name/\label in {{(-4,-2)/v_{17}/},{(-4,-4)/v_{21}/},{(-2,-4)/v_{22}/}}
         \node[vertex, label={[label distance=-1.5mm]45:{\footnotesize $\label$}}] (\name) at \pos {\footnotesize $\name$};
    \foreach \pos/\name/\label in {{(1,1)/v_{14}/2},{(-2,2)/v_{10}/8},{(-2,-2)/v_{18}/},{(1,-1)/v_{16}/1}}
         \node[vertex, label={[label distance=-1mm]right:{\footnotesize $\label$}}] (\name) at \pos {\footnotesize $\name$};
    \foreach \pos/\name/\label in {{(4,-2)/v_{20}/},{(2,-4)/v_{23}/},{(4,-4)/v_{24}/}}
         \node[vertex, label={[label distance=-1.5mm]135:{\footnotesize $\label$}}] (\name) at \pos {\footnotesize $\name$};
    \foreach \pos/\name/\label in {{(-2,4)/v_6/16},{(-4,2)/v_9/},{(-4,4)/v_5/32}}
         \node[vertex, label={[label distance=-1.5mm]315:{\footnotesize $\label$}}] (\name) at \pos {\footnotesize $\name$};
    \foreach \pos/\name/\label in { {(-5,-5)/v_3/},{(5,-5)/v_4/}}
         \node[vertex, label={[label distance=-0.5mm]below:{\footnotesize $\label$}}] (\name) at \pos {\footnotesize $\name$};
    \foreach \source/ \dest in {{v_1/v_2}, {v_2/v_4}, {v_4/v_3}, {v_3/v_1}, {v_3/v_{21}}, {v_4/v_{24}}, {v_2/v_8}, {v_1/v_5}, {v_5/v_6}, {v_6/v_7}, {v_7/v_8}, {v_8/v_{12}}, {v_{12}/v_{20}}, {v_{20}/v_{24}}, {v_{24}/v_{23}}, {v_{23}/v_{22}}, {v_{22}/v_{21}}, {v_{21}/v_{17}}, {v_{17}/v_9}, {v_9/v_5}, {v_9/v_{10}}, {v_6/v_{10}}, {v_{10}/v_{13}}, {v_7/v_{11}}, {v_{11}/v_{12}}, {v_{23}/v_{19}}, {v_{19}/v_{20}}, {v_{18}/v_{17}},{v_{22}/v_{18}},{v_{18}/v_{15}}, {v_{19}/v_{16}}, {v_{11}/v_{14}}, {v_{13}/v_{15}}, {v_{13}/v_{14}}, {v_{14}/v_{16}}, {v_{16}/v_{15}}}
        \path[edge] (\source) -- (\dest);
    \foreach \source/\dest in {{v_1/v_5},{v_5/v_6},{v_6/v_{10}},{v_{10}/v_{13}},{v_{13}/v_{14}},{v_{13}/v_{15}},{v_{15}/v_{16}}}
        \path[edge, line width=3pt] (\source) -- (\dest);
    \node[] at (6,0) {$T_6$};
\end{tikzpicture}
\end{subfigure}
    \caption{Set of tree strategies $\mathcal{T}_r = \{T_1, T_2, T_3, T_4, T_5, T_6\}$ that prove $\pi(B_4) \leq 66$.}
    \label{fig:bruhat-cert}
\end{figure}

\begin{theorem}
Let $B_4$ be the Bruhat graph of order $4$. Then $\pi(B_4) \leq 66$.
\end{theorem}
\begin{proof}
The $T= 6$ strategies with $r = v_1$ shown in Figure \ref{fig:bruhat-cert} certify the result. Since the graph is vertex transitive, only one root must be checked. 
\end{proof}
The bound was obtained on 16 threads with a walltime of 3 hours and maximum tree length set to $\ell = 16$. This is a significant improvement on the results presented in \cite{Hurlbert2010} and exhibits the usefulness and limitations of using linear optimization to generate tree strategies. Further, it raises the following conjecture. 
\begin{conjecture}
If $B_4$ is the 4\textsuperscript{th} weak Bruhat graph, then the bound $\pi(B_4) \leq 66$ is best possible using tree strategy weight functions. 
\end{conjecture}
We arrive at this conjecture after running $\mathbf{TS}_{B_4, r}$ with high values of $T$ and exhausting available computational resources. The strength of the proposed MILP approach is in its versatility and ability to generate a large number of tree strategy weight functions. We believe the bound off 66 is the best attainable through these means, not due to a restriction in the MILP approach but a limitation in the usefulness of tree strategy weight functions. We believe more general weight functions, such as those defined in \cite{Cranston2015}, have greater potential to improve the bound on $B_4$.

\subsection{Improving Results with Symmetry}  
\indent The Lemke square is of particular interest to pebbling researchers since it is a candidate for a counterexample to Graham's Conjecture \ref{conj:graham}. Due to the size and complexity of the Lemke square, generating tree strategies by hand is difficult and has many limitations. Hurlbert \cite{Hurlbert2010} used manually generated certificates to show that \begin{equation}\pi(L\Hsquare L, (r,r)) \leq \begin{cases} 108 &\text{ if } r = v_1 \\ 96 & \text{ if } r = v_8 \\ 68 &\text{ if } r = v_4,
\end{cases}\end{equation} using 4 strategies for $r=v_4$, 4 strategies for $r = v_8$ and 5 strategies for $r = v_1$. Furthermore, Kenter \cite{Kenter2019} uses an integer program to provide strong evidence that $\pi(L\Hsquare L) \leq 85$. Our goal is to use linear optimization to generate larger sets of strategies in hopes of improving these bounds and to provide certificates that serve to validate the proof. We apply both $\textbf{TS}_{G,r}$ and $\textbf{STS}_{G,r}$ to the Lemke square graph in this pursuit. In addition to successfully improving the bounds of Hurlbert \cite{Hurlbert2010}, the experimentation also provides useful insights into the computational properties of the linear programs. \\
\begin{table}[tb]
    \centering
    \begin{tabular}{|c|c|c|c|c|}\hline
         $T$ & $\ell$ & Walltime (sec)& Solver & Threads   \\\hline
         10 & 16 & 43200 & \texttt{Gurobi} & 16 \\\hline
    \end{tabular}
    \caption{Experimental parameters for $\textbf{STS}_{L\Hsquare L,r}$  with root $r = (v_1, v_1)$ that produced bound $\pi(L\Hsquare L, (v_1, v_1)) \leq 96$. }
    \label{tab:ls-params}
\end{table}
\begin{table}[tb]
    \centering
    \begin{tabular}{|c|c|c|c|c|}\hline
         $T$ & $\ell$ & Walltime (sec)& Solver & Threads   \\\hline
         8 & 16 & 10800 & \texttt{Gurobi} & 16 \\\hline
    \end{tabular}
    \caption{Experimental parameters for $\textbf{STS}_{L\Hsquare L,r}$ with root $r \neq (v_1, v_1)$ that produced bound $\pi(L\Hsquare L, r) \leq 96$ for all $r\in V(L\Hsquare L)$. }
    \label{tab:ls-all-params}
\end{table}
\indent Due to the size and complexity of the Lemke square, solving the MILPs above is computationally expensive. Given a graph with $n$ vertices and $m$ edges, generating a set of $T$ strategies involves $T(2n + m)$ total decision variables, $Tn$ of which are continuous and $T(n + m)$ of which are binary, and $T(2m +2n+1)+n$ constraints. As such, leveraging the symmetry of the Lemke square in $\mathbf{STS}_{L\Hsquare L,r}$ to reduce the number of decision variables and constraints yielded the lowest pebbling bound. \\
\indent After extensive experimentation, the set of parameters summarized in Tables \ref{tab:ls-params} and \ref{tab:ls-all-params} produced the solution closest to optimal. Proximity to optimality was measured by the optimality gap of the MILP solver, which ranged from $0.0\%$ for $r = (v_2, v_4)$ to $33.8\%$ for $r = (v_1,v_1)$. \footnote{The optimality gap of a MILP is the magnitude of the difference between the best integer solution (e.g. the incumbent solution) and best bound, divided by the best bound. Specifically, if $z_P$ is the incumbent objective value, and $z_D$ is the lower bound of the objective, the MILP gap is $|z_P - z_D|/|z_P|$. The best objective found for bounding $\pi(L\Hsquare L)$ was $63.0$ for each root.} As noted by Hurlbert \cite{Hurlbert2010}, the root $r = (v_1, v_1)$ proves to be the most difficult root for which to obtain low upper bound due to its low degree and location in the graph. The results in Table \ref{tab:ls-bounds} indicate that $(v_1, v_1)$ restricts the upper bound for all roots most tightly, as all other roots achieved a bound significantly lower than 96. Using $T= 10$ symmetric tree strategies (i.e. 20 total certificates) we prove that $\pi(L\Hsquare L, (v_1, v_1)) \leq 96$, and show that this bound holds for all other roots $r \neq (v_1, v_1)$ with $T= 6$ symmetric strategies. 
\begin{theorem}
Let $L$ be the Lemke graph. Then $\pi(L\Hsquare L, (r_1,r_2)) \leq 96$, for all $(r_1,r_2)\in L\Hsquare L$.
\end{theorem}
\begin{proof}
See \ref{subsec:app-cert} for certificates and \ref{subsec:app-viz} for the visualizations of each tree strategy with accompanying weight function for $r = (v_1, v_1)$. Certificates and visualizations for all $r \in V(L\Hsquare L) \backslash \{(v_1, v_1)\}$ are presented in the accompaning repository,\footnote{\url{https://github.com/dominicflocco/Graph_Pebbling}} and the bounds obtained for these vertices are presented in Table \ref{tab:ls-bounds}. The results were obtained using the parameters summarized in Table \ref{tab:ls-params} for $r = (v_1, v_1)$ and in Table \ref{tab:ls-all-params} for $r \neq (v_1, v_1)$. Without loss of generality, we compute bounds for the set of 32 roots presented in Table \ref{tab:ls-bounds}. By the symmetry of $L\Hsquare L$, $\pi(L\Hsquare L, (r_1, r_2)) = \pi(L\Hsquare L, (r_2, r_1))$, which provides bounds for the remaining 32 vertices. 
\end{proof}

\begin{table}[tb]
\centering
 \begin{threeparttable}[t]
    \centering
    \begin{tabular}{|c|c||c|c||c|c|}\hline
    $r$ & $\pi(L\Hsquare L, r) \leq$ &$r$ & $\pi(L\Hsquare L, r) \leq$  &$r$ & $\pi(L\Hsquare L, r) \leq$  \\\hline\hline
    
$(v_8, v_5)$&72&$(v_2, v_7)$&65&$(v_3, v_7)$&65\\\hline
$(v_5, v_2)$&65&$(v_1, v_8)$&64&$(v_7, v_5)$&67\\\hline
$(v_3, v_1)$&71&$(v_3, v_6)$&65&$(v_1, v_7)$&65\\\hline
$(v_6, v_2)$&65&$(v_6, v_4)$&66&$(v_1, v_6)$&65\\\hline
$(v_5, v_4)$&66&$(v_4, v_8)$&71&$(v_1, v_4)$&65\\\hline
$(v_1, v_2)$&75&$(v_6, v_8)$&71&$(v_2, v_4)$&64\\\hline
$(v_2, v_8)$&64&$(v_7, v_8)$&71&$(v_3, v_2)$&65\\\hline
$(v_6, v_5)$&67&$(v_6, v_7)$&66&$(v_7, v_7)$&72\\\hline
$(v_5, v_3)$&65&$(v_1, v_5)$&65&$(v_8, v_8)$&88\\\hline
$(v_5, v_5)$&75&$(v_3, v_4)$&65&$(v_4, v_4)$&66\\\hline
$(v_3, v_8)$&67&$(v_3, v_3)$&68&$(v_2, v_2)$&75\\\hline
$(v_4, v_7)$&66&$(v_6, v_6)$&70&$(v_1, v_1)$\tnote{a}&96\\\hline
    \end{tabular}
    
     \begin{tablenotes}
         \item [a] Using parameters outlined in Table \ref{tab:ls-params}.
         \end{tablenotes}
     \end{threeparttable}
     \caption{Upper bounds obtained with $\textbf{STS}_{L\Hsquare L,r}$ for $\pi(L\Hsquare L)$ for all roots $r\in V(L\Hsquare L) \backslash \{(v_1, v_1)\}$ using parameters outlined in Table \ref{tab:ls-all-params} unless noted otherwise.}
    \label{tab:ls-bounds}
\end{table}

We again see the usefulness of the MILP framework in automating the generation of weight functions to prove pebbling bounds on large graphs. Using linear optimization to produce large sets of tree strategies proved useful in improving the pebbling bounds on the Lemke square found in \cite{Hurlbert2010}. The method has the benefit of providing verifiable proofs, producing a large set of tree strategies and the adaptability to be applied to more general classes of graphs. Through automating the generation of tree strategies, we are able to consider large sets of these strategies and produce low pebbling bounds, as shown in Table \ref{tab:ls-bounds}. Note that the pebbling bound dropped below 96 for all roots $r \neq (v_1, v_1)$, and often very quickly (typically in less than 5 minutes using 32 threads). Further, these bounds were obtained using only 6 symmetric tree strategies. While the $\textbf{TS}_{L\Hsquare L, r}$ performed well on smaller graphs, it was unable to produce a bound better than $\pi(L\Hsquare L, (v_1, v_1)) \leq 98$ with 10 non-symmetric strategies. It should be noted that $\textbf{TS}_{L\Hsquare L, r}$  achieved the bound of 96 with 30 non-symmetric strategies. Nonetheless, the $\textbf{TS}_{L\Hsquare L, r}$ produced similar results to $\textbf{STS}_{L\Hsquare L, r}$ for roots $r \neq (v_1, v_1)$, displaying the potential to produce useful upper bounds. However, it is evident that leveraging the symmetry of graph products significantly improved performance on the Lemke square. \\ 
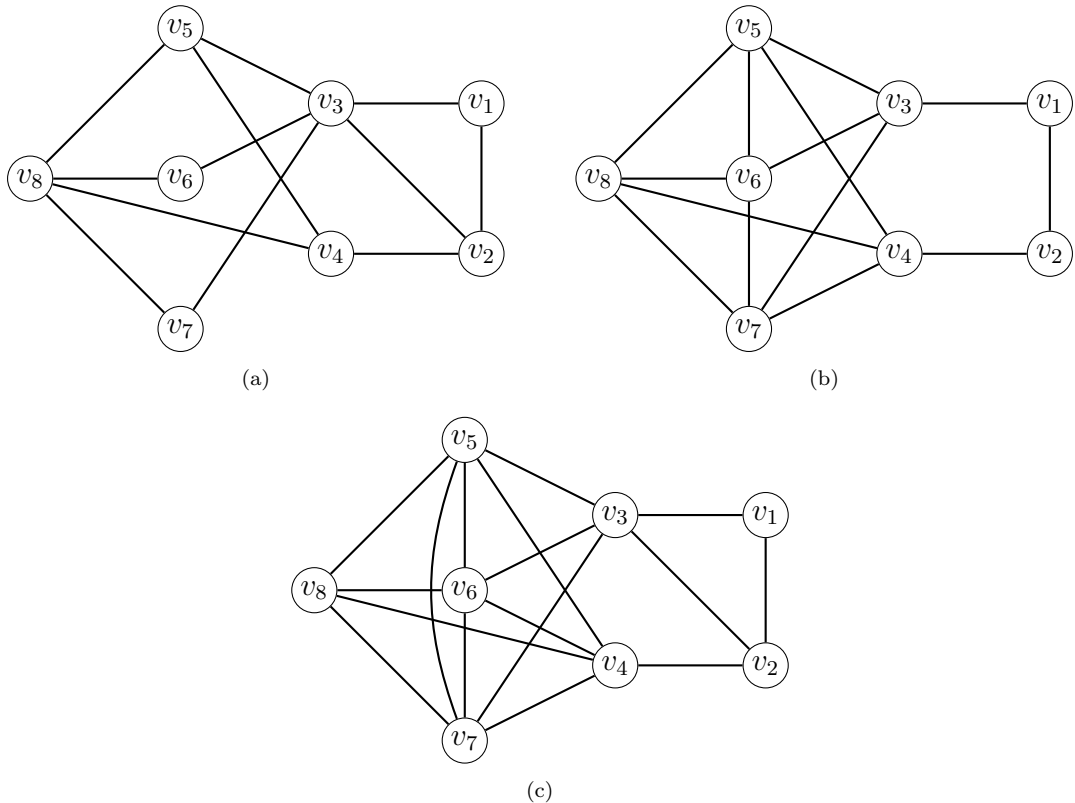
\begin{figure}[tb]
    \centering
    \begin{subfigure}{.5\textwidth}
    \centering
    \begin{tikzpicture}[scale=0.5, auto,swap, vertex/.style={circle,draw,minimum size=6mm,inner sep=0}]
    \foreach \pos/\name/\label in {{(0,0)/v_6/}, {(0,4)/v_5/}, {(0,-4)/v_7/}, {(4,2)/v_3/}, {(8,2)/v_1/}, {(4,-2)/v_4/}, {(8,-2)/v_2/}, {(-4,0)/v_8/}}
        \node[vertex, label={\footnotesize $\label$}] (\name) at \pos {$\name$};
    \foreach \source/ \dest in {{v_1/v_2},{v_1/v_3},{v_2/v_3},{v_2/v_4},{v_3/v_6},{v_3/v_5},{v_4/v_5},{v_3/v_7},{v_4/v_8},{v_5/v_8},{v_6/v_8},{v_7/v_8}}
        \path[edge] (\source) -- (\dest);
\end{tikzpicture}
    \caption{}
    \label{fig:lemke2-template}
\end{subfigure}%
\begin{subfigure}{.5\textwidth}
\centering
    \begin{tikzpicture}[scale=0.5, auto,swap,vertex/.style={circle,draw,minimum size=6mm,inner sep=0}]
    \foreach \pos/\name/\label in {{(0,0)/v_6/}, {(0,4)/v_5/}, {(0,-4)/v_7/}, {(4,2)/v_3/}, {(8,2)/v_1/}, {(4,-2)/v_4/}, {(8,-2)/v_2/}, {(-4,0)/v_8/}}
        \node[vertex, label={\footnotesize  $\label$}] (\name) at \pos {$\name$};
    \foreach \source/ \dest in {{v_1/v_2},{v_1/v_3},{v_2/v_4},{v_3/v_6},{v_3/v_5},{v_4/v_5},{v_4/v_7},{v_5/v_6},{v_6/v_7},{v_3/v_7},{v_4/v_8},{v_5/v_8},{v_6/v_8},{v_7/v_8}}
        \path[edge] (\source) -- (\dest);
\end{tikzpicture}
    \caption{}
    \label{fig:lemke3-template}
\end{subfigure}
\centering 
\begin{subfigure}{.5\textwidth}
\centering 
\begin{tikzpicture}[scale=0.5, auto,swap, vertex/.style={circle,draw,minimum size=6mm,inner sep=0}]
    \foreach \pos/\name/\label in {{(0,0)/v_6/}, {(0,4)/v_5/}, {(0,-4)/v_7/}, {(4,2)/v_3/}, {(8,2)/v_1/}, {(4,-2)/v_4/}, {(8,-2)/v_2/}, {(-4,0)/v_8/}}
        \node[vertex, label={\small $\label$}] (\name) at \pos {$\name$};
    \foreach \source/ \dest in {{v_8/v_5}, {v_8/v_7}, {v_8/v_6}, {v_7/v_3}, {v_7/v_4}, {v_6/v_3}, {v_6/v_4}, {v_5/v_3}, {v_5/v_4}, {v_4/v_2}, {v_3/v_1}, {v_2/v_1}, {v_8/v_4}, {v_2/v_3},{v_5/v_6},{v_6/v_7}}
        \path[edge] (\source) -- (\dest);
    \path[edge] (v_5) edge [bend right=20] (v_7);
\end{tikzpicture}
\caption{}
\label{fig:lemke4-template}
\end{subfigure}
\caption{The second and third minimal, and only maximal non-isomorphic Lemke graphs (a) $L_2$, (b) $L_3$, and (c) $L_4$, respectively, on $n=8$ vertices.}
\label{fig:lemke23-templates}
\end{figure}
\indent Upon closer examination of the tree strategies produced by $\textbf{STS}_{L\Hsquare L,r}$ for $r = (v_1, v_1)$, we notice that some generated strategies are larger and more complex than others. For example, strategies $T_4$ and $T_{10}$, shown in Figures \ref{fig:ls-viz4} and \ref{fig:ls-viz10}, respectively, are comprised of only one other vertex and the root. These trivial strategies suggests that the bound of 96 can be found with fewer strategies, however the optimization solver could not find a feasible solution to the MILP that produced this bound due to the problem's computational complexity. While the objective value plateaued at 95.1999 ($33.8\%$ optimal by MIP gap) after running on 16 threads for 18 hours, the MILP appears to have wasted  two strategies in its computation. After much experimentation on $L\Hsquare L$ using different software, problem parameters and computational power, we are convinced that the best attainable bound using tree strategies on the Lemke square is in fact 96.  \\ 
\indent To further exhibit the versatility of our MILP approach, we apply $\mathbf{STS}_{G,r}$ to all 3 minimal, and the one maximal, non-isomorphic Lemke square graphs on 8 vertices. We denote the additional 2 minimal non-isomorphic Lemke graphs as $L_2$ and $L_3$, and the maximal Lemke $L_4$, shown in Figure \ref{fig:lemke23-templates}. Using the parameters outlined in Table \ref{tab:ls-all-params}, we are able to compute an upper bound for the rooted pebbling number on $L_2 \Hsquare L_2$, $L_3 \Hsquare L_3$ and $L_4 \Hsquare L_4$ for all roots, and obtain the pebbling bounds $\pi(L_2 \Hsquare L_2) \leq 96$, $\pi(L_3 \Hsquare L_3) \leq 96$ and $\pi(L_4 \Hsquare L_4) \leq 96$. The results for each root are presented in Tables \ref{tab:ls2-bounds}, \ref{tab:ls3-bounds} and \ref{tab:ls4-bounds} in \ref{subsec:lemke23-results}, and the certificates are offered in the accompanying repository.\footnote{https://github.com/dominicflocco/Graph\_Pebbling} Not only do these results offer further insight into properties of Lemke square graphs, they also display the utility of our MILP approach.

\section{Conclusion}
\label{sec:conclusion}
Through our research, we present a number of automation approaches to weight function generation in graph pebbling which yield improved pebbling bounds. Namely, we offer a MILP formulation that generates a set of tree strategies that attempt to minimize the pebbling bound of a graph. Experiments show that this approach is capable of producing weight function certificates that offer lower pebbling bounds than can feasibly be generated by hand. Results indicate that large sets of tree strategies often yields lower pebbling bounds, which highlights the power of applying computational approaches to pebbling bound generation. Our MILP approach to tree strategy optimization is successful in improving the bound on the Bruhat graph to 66 using only tree strategies, confirming the results of \cite{Cranston2015} and improving results of \cite{Hurlbert2010}. This technique is also capable of producing tree strategies on smaller graphs in a short amount of time, indicating that the MILP will be successful when applied to more general graphs, possibly arising from applications. \\ 
\indent The research was initially motivated by Graham's Conjecture \ref{conj:graham}, leading to the investigation of graph products and the Lemke square as a potential counterexample. The benefit of our approach is the generation of certificates that can be verified by hand and bypass the potential of precision error in arithmetic. Additionally, the MILP is a general purpose framework that can be applied to any connected graph, regardless of its structure or pebbling properties. Through leveraging the symmetry of Cartesian product graphs, we were able to produce a set of tree strategy certificates that prove $\pi(L\Hsquare L , r) \leq 96$ for all $r \in V(L\Hsquare L)$, which improves the results of Hurlbert \cite{Hurlbert2010}.  We also show that the upper bound of $\pi(L\Hsquare L , r)$ drops close to 64, and occasionally to equality, for some roots. These results exhibit the interesting properties of the Lemke square and its potential to be a counterexample for Graham's Conjecture. 

\section*{Acknowledgements}
The authors would like to thank the reviewer for their diligent and thoughtful review of the manuscript. 

\bibliographystyle{elsarticle-num} 
\bibliography{pebbling}

\newpage 
\appendix
\section{Appendix}
\label{sec:appendix}

\subsection{Rooted pebbling number bounds on other non-isomorphic Lemke graphs.}
\label{subsec:lemke23-results}
\def\arraystretch{1.2}%
\begin{table}[!ht]
\centering
 \begin{threeparttable}[t]
    \centering
    \begin{tabular}{|c|c||c|c||c|c|}\hline
    $r$ & $\pi(L_2\Hsquare L_2, r) \leq$ &$r$ & $\pi(L_2\Hsquare L_2, r) \leq$  &$r$ & $\pi(L_2\Hsquare L_2, r) \leq$  \\\hline\hline
    
$(v_8, v_5)$&72&$(v_2, v_7)$&65&$(v_3, v_7)$&65\\\hline
$(v_5, v_2)$&65&$(v_1, v_8)$&64&$(v_7, v_5)$&67\\\hline
$(v_3, v_1)$&71&$(v_3, v_6)$&65&$(v_1, v_7)$&67\\\hline
$(v_6, v_2)$&65&$(v_6, v_4)$&67&$(v_1, v_6)$&68\\\hline
$(v_5, v_4)$&68&$(v_4, v_8)$&71&$(v_1, v_4)$&66\\\hline
$(v_1, v_2)$&75&$(v_6, v_8)$&71&$(v_2, v_4)$&67\\\hline
$(v_2, v_8)$&64&$(v_7, v_8)$&71&$(v_3, v_2)$&65\\\hline
$(v_6, v_5)$&68&$(v_6, v_7)$&70&$(v_7, v_7)$&73\\\hline
$(v_5, v_3)$&64&$(v_1, v_5)$&65&$(v_8, v_8)$\tnote{a}&87\\\hline
$(v_5, v_5)$\tnote{a}&69&$(v_3, v_4)$&64&$(v_4, v_4)$&74\\\hline
$(v_3, v_8)$&65&$(v_3, v_3)$&65&$(v_2, v_2)$&71\\\hline
$(v_4, v_7)$&67&$(v_6, v_6)$&78&$(v_1, v_1)$\tnote{b}&96\\\hline
    \end{tabular}

     \begin{tablenotes}
         \item [a] Walltime 21,600 seconds.
         \item [b] Walltime 43,200 seconds.
         \end{tablenotes}
     \end{threeparttable}
     \caption{Upper bounds obtained with $\textbf{STS}_{L_2\Hsquare L_2,r}$ for $\pi(L_2\Hsquare L_2)$ for all roots $r\in V(L_2\Hsquare L_2)$ using parameters outlined in Table \ref{tab:ls-all-params} unless noted otherwise.}
    \label{tab:ls2-bounds}
\end{table}
\begin{table}[tb]
\centering
\def\arraystretch{1.2}%
 \begin{threeparttable}[t]
    \centering
    \begin{tabular}{|c|c||c|c||c|c|}\hline
    $r$ & $\pi(L_3\Hsquare L_3, r) \leq$ &$r$ & $\pi(L_3\Hsquare L_3, r) \leq$  &$r$ & $\pi(L_3\Hsquare L_3, r) \leq$  \\\hline\hline
$(v_8, v_5)$&71&$(v_2, v_7)$&65&$(v_3, v_7)$&66\\\hline
$(v_5, v_2)$&64&$(v_1, v_8)$&65&$(v_7, v_5)$&67\\\hline
$(v_3, v_1)$&71&$(v_3, v_6)$&71&$(v_1, v_7)$&65\\\hline
$(v_6, v_2)$&65&$(v_6, v_4)$&66&$(v_1, v_6)$&65\\\hline
$(v_5, v_4)$&66&$(v_4, v_8)$&71&$(v_1, v_4)$&64\\\hline
$(v_1, v_2)$&78&$(v_6, v_8)$&72&$(v_2, v_4)$&71\\\hline
$(v_2, v_8)$&65&$(v_7, v_8)$&72&$(v_3, v_2)$&65\\\hline
$(v_6, v_5)$&72&$(v_6, v_7)$&72&$(v_7, v_7)$&70\\\hline
$(v_5, v_3)$&66&$(v_1, v_5)$&65&$(v_8, v_8)$\tnote{a}&90\\\hline
$(v_5, v_5)$&69&$(v_3, v_4)$&65&$(v_4, v_4)$\tnote{a}&69\\\hline
$(v_3, v_8)$&65&$(v_3, v_3)$&69&$(v_2, v_2)$&67\\\hline
$(v_4, v_7)$&66&$(v_6, v_6)$&91&$(v_1, v_1)$\tnote{b}&96\\\hline

    \end{tabular}
    
     \begin{tablenotes}
         \item [a] Walltime 21,600 seconds.
          \item [b] Walltime 43,200 seconds.
         \end{tablenotes}
     \end{threeparttable}
     \caption{Upper bounds obtained with $\textbf{STS}_{L_3\Hsquare L_3,r}$ for $\pi(L_3\Hsquare L_3)$ for all roots $r\in V(L_3\Hsquare L_3)$ using parameters outlined in Table \ref{tab:ls-all-params} unless noted otherwise.}
    \label{tab:ls3-bounds}
\end{table}

\def\arraystretch{1.2}%
\begin{table}[tb]
\centering
 \begin{threeparttable}[t]
    \centering
    \begin{tabular}{|c|c||c|c||c|c|}\hline
    $r$ & $\pi(L_4\Hsquare L_4, r) \leq$ &$r$ & $\pi(L_4\Hsquare L_4, r) \leq$  &$r$ & $\pi(L_4\Hsquare L_4, r) \leq$  \\\hline\hline
    
$(v_7, v_4)$&68&$(v_3, v_4)$&64&$(v_7, v_7)$&69\\ \hline$(v_3, v_7)$&64&$(v_5, v_6)$&68&$(v_3, v_6)$&64\\ \hline$(v_2, v_8)$&65&$(v_1, v_3)$&76&$(v_8, v_8)$\tnote{a}&90\\ \hline$(v_3, v_5)$&64&$(v_4, v_2)$&64&$(v_4, v_5)$&68\\ \hline$(v_6, v_7)$&68&$(v_2, v_3)$&65&$(v_1, v_8)$&64\\ \hline$(v_6, v_1)$&66&$(v_3, v_3)$&65&$(v_5, v_7)$&68\\ \hline$(v_5, v_5)$&69&$(v_1, v_1)$\tnote{b}&96&$(v_2, v_5)$&64\\ \hline$(v_6, v_8)$&76&$(v_8, v_7)$&72&$(v_5, v_1)$&66\\ \hline$(v_6, v_4)$&68&$(v_2, v_6)$&64&$(v_2, v_7)$&65\\ \hline$(v_2, v_2)$&71&$(v_3, v_8)$&64&$(v_4, v_8)$&73\\ \hline$(v_1, v_2)$&80&$(v_1, v_4)$&66&$(v_7, v_1)$&66\\ \hline$(v_4, v_4)$&69&$(v_5, v_8)$&76&$(v_6, v_6)$&69\\ \hline
    \end{tabular}

     \begin{tablenotes}
         \item [a] Walltime 21,600 seconds.
         \item [b] Walltime 43,200 seconds.
         \end{tablenotes}
     \end{threeparttable}
     \caption{Upper bounds obtained with $\textbf{STS}_{L_4\Hsquare L_4,r}$ for $\pi(L_4\Hsquare L_4)$ for all roots $r\in V(L_4\Hsquare L_4)$ using parameters outlined in Table \ref{tab:ls-all-params} unless noted otherwise.}
    \label{tab:ls4-bounds}
\end{table}
\clearpage

\subsection{Lemke Square Certificates}
\label{subsec:app-cert}

\begin{table}[!h]

\begin{minipage}{.5\linewidth}

    \centering
    
    \resizebox{\linewidth}{!}{%
    \begin{tabular}{|c|c|c|c|c|c|c|c|c|} \hline
        
    $T_1$&$v_1$&$v_2$&$v_3$&$v_4$&$v_5$&$v_6$&$v_7$&$v_8$ \\\hline
$v_1$&&&40.00&&&&&\\\hline
$v_2$&&&20.00&1.88&10.00&10.00&10.00&1.50\\\hline
$v_3$&&&&&&&&\\\hline
$v_4$ &&&9.00&1.50&5.00&5.00&5.00&2.50\\\hline
$v_5$&&&&&2.50&2.25&&1.25\\\hline
$v_6$&&&3.75&&2.50&&2.50&1.25\\\hline
$v_7$&&&&1.25&2.50&1.37&2.50&1.25\\\hline
$v_8$&&&&1.00&1.25&2.50&2.50&1.25\\\hline
    \end{tabular}%
    }
    \label{tab:ls-cert1}
\end{minipage}%
\begin{minipage}{.5\linewidth}
    \centering
    \resizebox{.95\linewidth}{!}{%
    \begin{tabular}{|c|c|c|c|c|c|c|c|c|} \hline
        
    $T_2$&$v_1$&$v_2$&$v_3$&$v_4$&$v_5$&$v_6$&$v_7$&$v_8$\\\hline
$v_1$&&32.00&&&&&&\\\hline
$v_2$&&16.00&&&&&&\\\hline
$v_3$&&&&&&&&\\\hline
$v_4$ &&8.00&&4.00&1.50&&&\\\hline
$v_5$&&1.00&&&&&&\\\hline
$v_6$&&&&2.00&&1.00&1.00&1.00\\\hline
$v_7$&&&&&&&&\\\hline
$v_8$&&&&2.00&&&1.00&1.00\\\hline
    \end{tabular}%
    }
    \label{tab:ls-cert2}
    \end{minipage} 
    \caption{Certificates $T_1$ and $T_2$ for $\pi(L\Hsquare L) \leq 96$ proof.}
\end{table}

\begin{table}[!ht]
\begin{minipage}{.5\linewidth}

    \centering
    
    \resizebox{.82\linewidth}{!}{%
    \begin{tabular}{|c|c|c|c|c|c|c|c|c|} \hline
    $T_3$&$v_1$&$v_2$&$v_3$&$v_4$&$v_5$&$v_6$&$v_7$&$v_8$\\\hline
$v_1$&&1.00&&&&&&\\\hline
$v_2$&&&&&&&&\\\hline
$v_3$&&&&&&&&\\\hline
$v_4$ &&&&&&&&\\\hline
$v_5$&&&&&&&&\\\hline
$v_6$&&&&&&&&\\\hline
$v_7$&&&&&&&&\\\hline
$v_8$&&&&&&&&\\\hline

\end{tabular}%
    }

    \label{tab:ls-cert3}
    \end{minipage}%
\begin{minipage}{.5\linewidth}
    \centering
    \resizebox{\linewidth}{!}{%

    \centering
    \begin{tabular}{|c|c|c|c|c|c|c|c|c|} \hline
        
    $T_4$&$v_1$&$v_2$&$v_3$&$v_4$&$v_5$&$v_6$&$v_7$&$v_8$\\\hline
$v_1$&&&&&&&&\\\hline
$v_2$&&&&&&&&\\\hline
$v_3$&32.00&&&&&&&\\\hline
$v_4$ &&&&&&&&\\\hline
$v_5$&16.00&8.00&3.00&4.00&2.00&&&\\\hline
$v_6$&&&&1.00&&&&\\\hline
$v_7$&&&&&&&&\\\hline
$v_8$&2.00&&1.00&2.00&1.00&&&1.00\\\hline
    \end{tabular}%
    }

    \label{tab:ls-cert4}
\end{minipage}
    \caption{Certificates $T_3$ and $T_4$ for $\pi(L\Hsquare L) \leq 96$ proof.}
\end{table}

\begin{table}[!ht]

\begin{minipage}{.5\linewidth}

    \centering
    
    \resizebox{.92\linewidth}{!}{%
    \begin{tabular}{|c|c|c|c|c|c|c|c|c|} \hline
        
    $T_5$&$v_1$&$v_2$&$v_3$&$v_4$&$v_5$&$v_6$&$v_7$&$v_8$\\\hline
$v_1$&&&36.00&&&&&\\\hline
$v_2$&&&&&&&&\\\hline
$v_3$&&&18.00&&9.00&&9.00&1.00\\\hline
$v_4$ &&&&&1.00&&2.25&\\\hline
$v_5$&&&&&&&4.50&\\\hline
$v_6$&&&2.25&2.00&4.50&1.00&4.50&2.25\\\hline
$v_7$&&&7.50&&4.50&1.00&3.00&2.25\\\hline
$v_8$&&&&1.00&2.25&&&1.13\\\hline
    \end{tabular}%
    }
    \label{tab:ls-cert5}
\end{minipage}%
\begin{minipage}{.5\linewidth}
    \centering
    \resizebox{\linewidth}{!}{%
    \begin{tabular}{|c|c|c|c|c|c|c|c|c|} \hline
        
    $T_6$&$v_1$&$v_2$&$v_3$&$v_4$&$v_5$&$v_6$&$v_7$&$v_8$\\\hline
$v_1$&&&&&&&&\\\hline
$v_2$&&&&&&&&\\\hline
$v_3$&56.00&&&&&&3.50&\\\hline
$v_4$ &&&&&&&&1.75\\\hline
$v_5$&&1.00&&&&&&\\\hline
$v_6$&28.00&10.00&14.00&1.25&7.00&7.00&7.00&3.50\\\hline
$v_7$&&&&&&&1.00&\\\hline
$v_8$&&5.00&5.00&&3.13&3.50&3.50&1.75\\\hline
    \end{tabular}%
    }
    \label{tab:ls-cert6}
\end{minipage}
\caption{Certificates $T_5$ and $T_6$ for $\pi(L\Hsquare L) \leq 96$ proof.}
    
\end{table}

\begin{table}[!ht]
    \begin{minipage}{.5\linewidth}

    \centering
    
    \resizebox{.95\linewidth}{!}{%
    \begin{tabular}{|c|c|c|c|c|c|c|c|c|} \hline
        
    $T_7$&$v_1$&$v_2$&$v_3$&$v_4$&$v_5$&$v_6$&$v_7$&$v_8$\\\hline
$v_1$&&&48.00&&&&24.00&12.00\\\hline
$v_2$&&&&&&&10.00&6.00\\\hline
$v_3$&&&&&&&&6.00\\\hline
$v_4$ &&&&&&&5.00&3.00\\\hline
$v_5$&&&&&&&&3.00\\\hline
$v_6$&&&&&&&&3.00\\\hline
$v_7$&&&&1.00&&&2.50&3.00\\\hline
$v_8$&&&1.00&&&&2.25&1.50\\\hline
    \end{tabular}%
    }
    \label{tab:ls-cert7}
    \end{minipage}%
\begin{minipage}{.5\linewidth}
     \centering
    \resizebox{\linewidth}{!}{%
     \begin{tabular}{|c|c|c|c|c|c|c|c|c|} \hline
        
    $T_8$&$v_1$&$v_2$&$v_3$&$v_4$&$v_5$&$v_6$&$v_7$&$v_8$\\\hline
$v_1$&&&32.00&&16.00&&&6.00\\\hline
$v_2$&&&&&&&&3.00\\\hline
$v_3$&&&&&8.00&&&4.00\\\hline
$v_4$ &&&&&1.00&&&1.00\\\hline
$v_5$&&&&2.00&4.00&&&2.00\\\hline
$v_6$&&&&1.00&&1.00&1.00&2.00\\\hline
$v_7$&&&&&4.00&&1.00&2.00\\\hline
$v_8$&&&&1.00&2.00&&&1.00\\\hline
    \end{tabular}%
    } 
    \label{tab:ls-cert8}
    \end{minipage}
    \caption{Certificates $T_7$ and $T_8$ for $\pi(L\Hsquare L) \leq 96$ proof.}
\end{table}

\begin{table}[!ht]
    \begin{minipage}{.5\linewidth}

    \centering
    
    \resizebox{\linewidth}{!}{%
    \begin{tabular}{|c|c|c|c|c|c|c|c|c|} \hline
        
    $T_9$&$v_1$&$v_2$&$v_3$&$v_4$&$v_5$&$v_6$&$v_7$&$v_8$\\\hline
$v_1$&&44.00&&22.00&&&&\\\hline
$v_2$&&&&10.12&&&&4.50\\\hline
$v_3$&&&&11.00&&&&2.00\\\hline
$v_4$ &&&&4.50&&2.25&&2.00\\\hline
$v_5$&&&&5.50&1.50&1.00&1.75&2.75\\\hline
$v_6$&&&&5.50&2.75&&&1.00\\\hline
$v_7$&&&&5.50&2.75&1.63&&2.25\\\hline
$v_8$&&&&2.75&1.37&&&1.37\\\hline
    \end{tabular}%
    } 
    \label{tab:ls-cert9}
    \end{minipage}%
\begin{minipage}{.5\linewidth}
    \centering
    \resizebox{.85\linewidth}{!}{%
    \begin{tabular}{|c|c|c|c|c|c|c|c|c|} \hline
        
$T_{10}$&$v_1$&$v_2$&$v_3$&$v_4$&$v_5$&$v_6$&$v_7$&$v_8$\\\hline
$v_1$&&&&&&&&\\\hline
$v_2$&1.00&&&&&&&\\\hline
$v_3$&&&&&&&&\\\hline
$v_4$ &&&&&&&&\\\hline
$v_5$&&&&&&&&\\\hline
$v_6$&&&&&&&&\\\hline
$v_7$&&&&&&&&\\\hline
$v_8$&&&&&&&&\\\hline
    \end{tabular}%
    }
    \label{tab:ls-cert10}
    \end{minipage}
    \caption{Certificates $T_9$ and $T_{10}$ for $\pi(L\Hsquare L) \leq 96$ proof.}
\end{table}

\clearpage

\subsection{Lemke Square Strategy Visualizations}
\label{subsec:app-viz}
In Figures \ref{fig:ls-viz1}-\ref{fig:ls-viz10} below, labels $(i,j)$ on top of each node correspond to vertex $(v_{i+1}, v_{j+1})\in V(L\Hsquare L)$ used in Figure \ref{fig:lemke-template}, and labels on the bottom of each node represent vertex weights.

\begin{figure}[!ht]
    \centering
    \includegraphics[scale = 0.50]{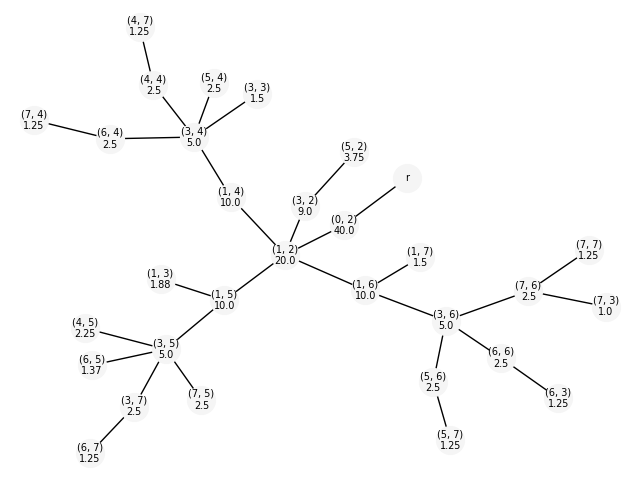}
    \caption{Visualization of $T_1$ for $\pi(L\Hsquare L) \leq 96$ proof.}
    \label{fig:ls-viz1}
\end{figure}

\begin{figure}[!ht]
    \centering
    \includegraphics[scale = 0.50]{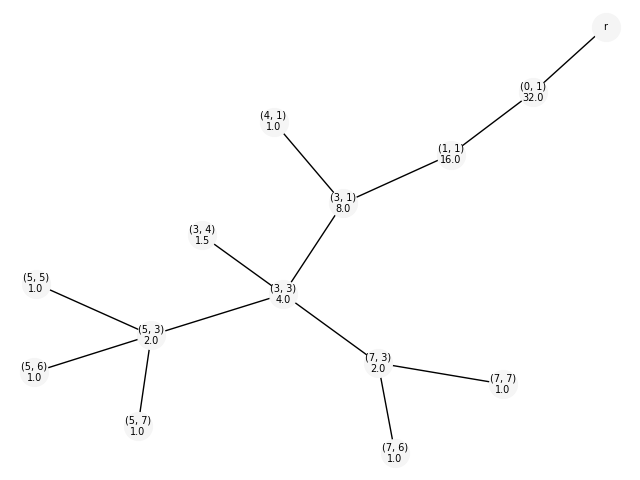}
     \caption{Visualization of $T_2$ for $\pi(L\Hsquare L) \leq 96$ proof.}
    \label{fig:ls-viz2}
\end{figure}

\begin{figure}[!ht]
    \centering
    \includegraphics[scale = 0.70]{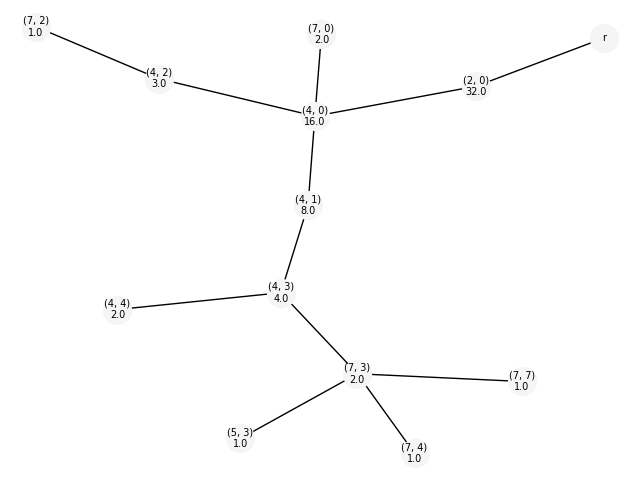}
     \caption{Visualization of $T_3$ for $\pi(L\Hsquare L) \leq 96$ proof.}
    \label{fig:ls-viz3}
\end{figure}

\begin{figure}[!ht]
    \centering
    \includegraphics[scale = 0.70]{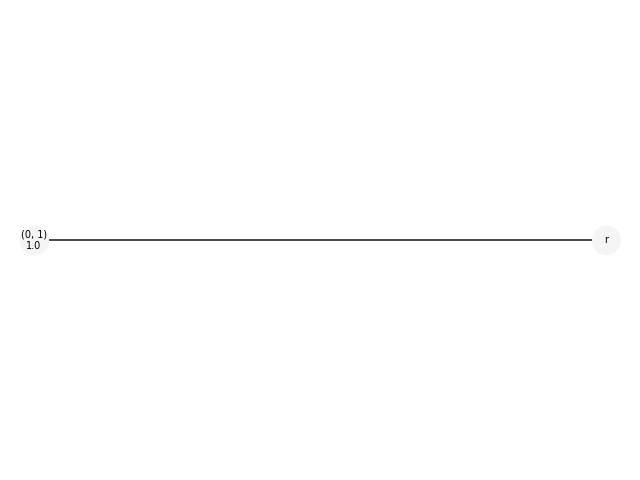}
     \caption{Visualization of $T_4$ for $\pi(L\Hsquare L) \leq 96$ proof.}
    \label{fig:ls-viz4}
\end{figure}

\begin{figure}[!ht]
    \centering
    \includegraphics[scale = 0.70]{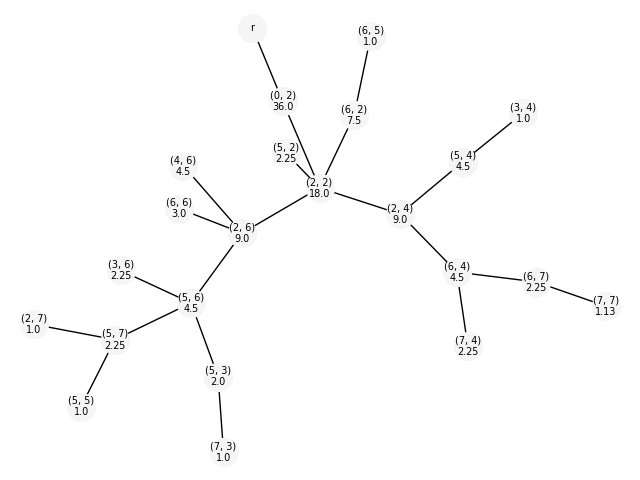}
     \caption{Visualization of $T_5$ for $\pi(L\Hsquare L) \leq 96$ proof.}
    \label{fig:ls-viz5}
\end{figure}

\begin{figure}[!ht]
    \centering
    \includegraphics[scale = 0.7]{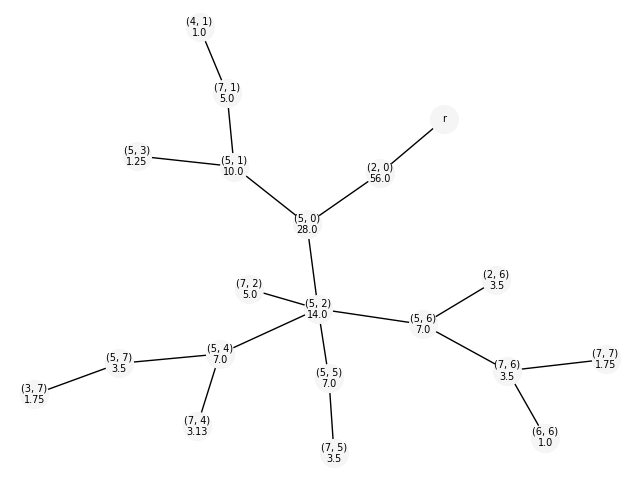}
     \caption{Visualization of $T_6$ for $\pi(L\Hsquare L) \leq 96$ proof.}
    \label{fig:ls-viz6}
\end{figure}

\begin{figure}[!ht]
    \centering
    \includegraphics[scale = 0.7]{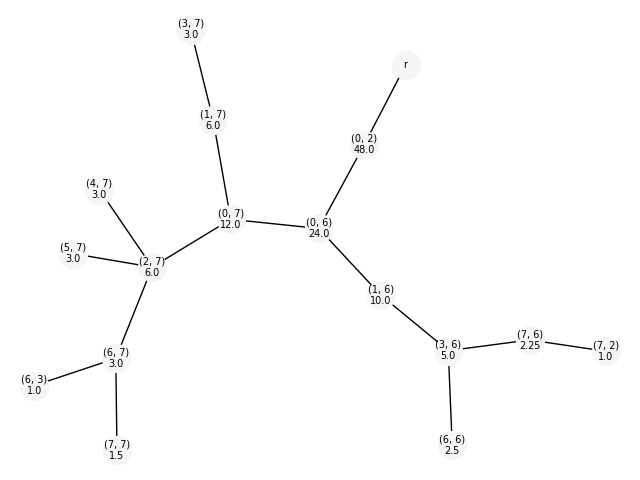}
    \caption{Visualization of $T_7$ for $\pi(L\Hsquare L) \leq 96$ proof.}
    \label{fig:ls-viz7}
\end{figure}

\begin{figure}[!ht]
    \centering
    \includegraphics[scale = 0.7]{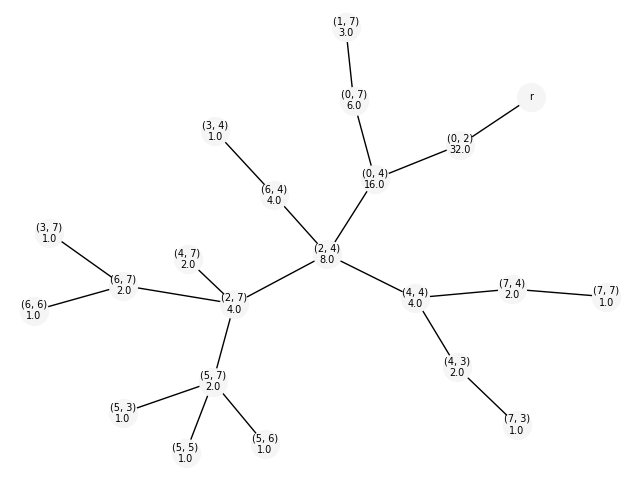}
     \caption{Visualization of $T_8$ for $\pi(L\Hsquare L) \leq 96$ proof.}
    \label{fig:ls-viz8}
\end{figure}

\begin{figure}[!ht]
    \centering
    \includegraphics[scale = 0.7]{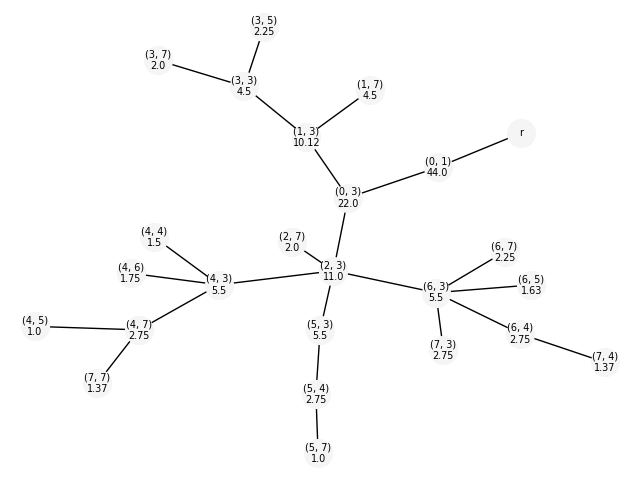}
    \ \caption{Visualization of $T_9$ for $\pi(L\Hsquare L) \leq 96$ proof.}
    \label{fig:ls-viz9}
\end{figure}

\begin{figure}[!ht]
    \centering
    \includegraphics[scale = 0.7]{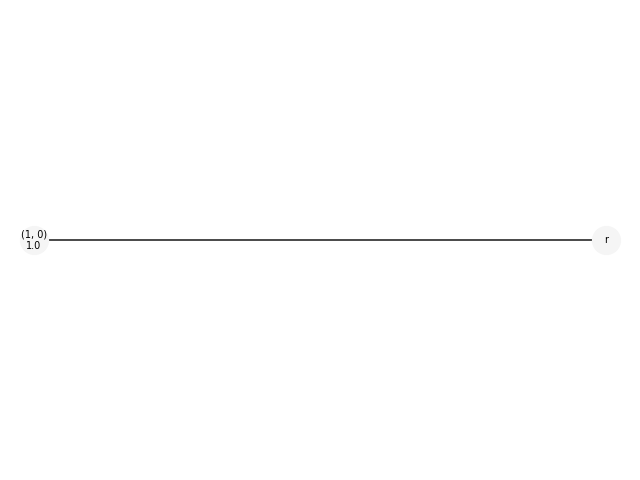}
     \caption{Visualization of $T_{10}$ for $\pi(L\Hsquare L) \leq 96$ proof.}
    \label{fig:ls-viz10}
\end{figure}





\end{document}